\documentclass[a4paper]{amsart}

\usepackage[all]{xy}
\usepackage{amsmath,amssymb,amsfonts}
\usepackage[mathscr]{euscript}

\theoremstyle{plain}
\newtheorem{theorem}{Theorem}[section]
\newtheorem{lemma}[theorem]{Lemma}
\newtheorem{proposition}[theorem]{Proposition}
\newtheorem{corollary}[theorem]{Corollary}
\newtheorem{definition}[theorem]{Definition}

\newtheorem{example}[theorem]{Example}
\newtheorem{remark}[theorem]{Remark}
\newtheorem{question}{Question}

\def\<#1>{\langle\, #1\,\rangle}

\newcommand{\restr}{\hskip-4pt \restriction}

\newcommand{\R}{\mathbb{R}}
\newcommand{\C}{\mathbb{C}}
\newcommand{\N}{\mathbb{N}}

\newcommand{\B}{\mathscr{B}}
\newcommand{\A}{\mathscr{A}}

\newcommand{\Z}{\mathbb{Z}}

\newcommand{\cc}{\mathfrak{c}}
\newcommand{\cl}{\mathrm{cl}}

\newcommand{\cnaught}{{C_0(G)}}

\newcommand{\luc}{\mathit{LUC}(G)}

\renewcommand{\emptyset}{\varnothing}

\font\seis=cmr6
\def\CB{\mathscr{CB}}
\def\ruc{{\seis{\mathscr{RUC}}}}
\def\uc{\mathscr{UC}}
\def\luc{{\seis{\mathscr{LUC}}}}
\def\wap{{\seis{\mathscr{WAP}}}}
\def\lc{{\seis{\mathscr{LC}}}}
\def\B{{\mathscr{B}}}

\def\ap{{\seis{\mathscr{AP}}}}

\begin{document}

\title{
Approximable $\wap$- and $\luc$-interpolation sets}

\author[Filali and Galindo]{M. Filali \and  J. Galindo}
\thanks{Research of the first author partially supported by Grant
INV-2010-20 of the 2010 Program for Visiting Researchers of
University Jaume I}
\thanks{Research of  the second author  partially supported by the Spanish
Ministry of Science (including FEDER funds), grant MTM2011-23118 and Fundaci\'o Caixa Castell\'o-Bancaixa, grant P1.1B2008-26}
\keywords{uniformly continuous functions, weakly almost periodic,
semigroup compactifications, approximable interpolation sets,
uniformly discrete sets, translation-compact sets, strongly prime
points. }

\address{\noindent Mahmoud Filali,
Department of Mathematical Sciences\\University of Oulu\\Oulu,
Finland. \hfill\break \noindent E-mail: {\tt mfilali@cc.oulu.fi}}
\address{\noindent Jorge Galindo, Instituto Universitario de Matem\'aticas y
Aplicaciones (IMAC)\\ Universidad Jaume I, E-12071, Cas\-tell\'on,
Spain. \hfill\break \noindent E-mail: {\tt jgalindo@mat.uji.es}}

\subjclass[2010]{Primary 43A46; Secondary 22D15, 43A15, 43A60,
54H11}

\date{\today}

\begin{abstract} Extending and unifying concepts
 extensively used in the literature, we introduce  the notion of approximable interpolation sets
 for  algebras of functions on locally compact groups,
 especially for weakly almost periodic functions and for uniformly
  continuous functions. We characterize approximable interpolation
  sets both in combinatorial terms and in terms of the $\luc$- and
$\wap$-compactifications and analyze some of their properties.
\end{abstract}

\maketitle
\tableofcontents
\section{Introduction}

Interpolation sets have been a key technique for the construction of
functions of various types on infinite discrete or, more generally,
locally compact groups. They have the crucial property that any bounded
function defined on them extends to the whole group as a function of
the required type.

If we require the extended functions to be almost periodic, then
interpolation sets are usually known as $I_0$-sets and were
introduced by Hartman and Ryll-Nardzewsky \cite{hartryll64}. For
further details and recent results on $I_0$-sets, see for example
the papers by Galindo and Hern\'andez  \cite{galihern99fu,GH} Graham
and Hare \cite{grahhare06i,grahhare05iii,grahhare06iv}, Graham, Hare and
K\"orner \cite{grahharekornii} or  Hern\'andez \cite{hern08}.

Interpolation sets for the functions in the Fourier-Stieltjes
algebra $B(G)$ are usually known as Sidon sets when the group $G$ is discrete and
Abelian and weak Sidon sets in general, see for instance the works
by Lopez and Ross  \cite{lopezross}, and Picardello \cite{pica73b}.
Sidon sets  are in fact uniformly approximable  as proved by Drury in
\cite{drury70}. This means that in addition of being interpolation
sets for the Fourier-Stieltjes algebra, the characteristic function
of the set can be uniformly approximated by members of the algebra;
in other words, the characteristic function
of the set  belongs to the Eberlein algebra
$\B(G)=\overline{B(G)}^{\|\cdot\|_\infty}$. This fact has important
consequences as for instance Drury's  union theorem: the union of
two Sidon subsets of a discrete Abelian group remains Sidon.

 Ruppert  \cite{rup} and Chou \cite{C4} considered interpolation sets
 for the algebra of  weakly almost periodic functions on discrete groups and semigroups, again  with the extra condition that the
characteristic function of the set is weakly almost periodic. This
is equivalent to the property that all bounded functions vanishing off the set being
weakly almost periodic. These interpolation sets were called
\emph{translation-finite} sets (after their combinatorial
characterization) by Ruppert and $R_W$-sets (after their
interpolation properties) by Chou.

Let $\ell_\infty(G)$ be the C*-algebra of bounded,  scalar-valued
functions on $G$ with the supremum norm and let $\A(G)\subseteq
\ell_\infty(G)$. In the present paper we introduce  the notion of
approximable $\A(G)$-interpolation sets in such a way that
it is suitable for functions defined on any topological group $G$
and reduces to that of uniformly approximable Sidon sets when $G$ is
discrete and $\A(G)=B(G)$, and to that of $R_W$-sets or
translation-finite sets when  $G$ is discrete and $\A=\wap(G),$ the algebra of
weakly almost periodic functions on $G.$ Since we shall be dealing
with closed subalgebras of $\ell_\infty(G)$, save some brief
digressions around $B(G)$, we shall omit the adverb "uniformly" in
our definition (cf. Definition \ref{def:main}).

As the reader might expect,
 approximable $\A(G)$-interpolation sets can be found in
  abundance if the algebra is large while they might be hard
to find if the algebra is too small.  As extreme cases, we could
mention that all subsets of $G$ have the property if
$\A(G)=\ell_\infty(G)$, while no metrizable locally compact group can
have infinite approximable
 $\ap(G)$-interpolation sets, where $\ap(G)$ is the algebra of almost periodic functions on
 $G$ (see below).

Our principal concern shall be with the algebras $\luc(G)$, of bounded
functions which are uniformly continuous with respect to the right
uniformity of $G$, and  $\wap(G)$, of weakly almost
periodic functions on $G$. A combinatorial  characterization of
approximable
 $\luc(G)$- and $\wap(G)$-interpolation sets will be
presented in Section 4.

But beforehand, we deal in Section 3  with the more straightforward
cases of the algebra $\CB(G)$ of bounded, continuous, scalar-valued
functions and the algebra $\cnaught$ of continuous functions
vanishing at infinity on $G.$ It turns out that for the algebras
$\cnaught,$ $\CB(G)$ and $\luc(G)$, interpolation sets and
approximable interpolation sets are the same for any topological group.
This is also true for the algebra $\B(G)$ when $G$ is an Abelian
discrete group, a fact that follows from Drury's theorem since, as
we shall  see in Proposition \ref{openpb},
 uniformly approximable Sidon sets  (in the sense of Dunkl-Ramirez \cite{dunklrami}) are the same as our approximable $\B(G)$-interpolation sets when $G$ is discrete.

We do not know if this stays true for any locally compact group, but
we give a partial result towards its affirmation for locally compact
metrizable groups. It should be remarked that approximability cannot
be expected within  the algebra $B(G)$. In fact,   Dunkl and Ramirez
noted in \cite[Remark 5.5, page 59]{dunklrami} that if $T$ is a Sidon
set in a discrete Abelian group $G$, and $1_T$ denotes the
characteristic function of $T$, then $1_T\in
 B(G)$ if and only if $T$ is finite.  They observed also that for $G=\mathbb Z,$ this holds for any subset of $G$.

It is quick to see that the  closed discrete sets are
$\CB(G)$-interpolation sets when the topological space underlying
$G$ is normal, and the finite sets are $\cnaught$-interpolation sets (Proposition \ref{cb}).
In Section 4, we see that  the right (left) uniformly discrete sets
are approximable interpolation sets for the algebra  $\luc(G)$ ($\ruc(G)$) for any topological
group. The converse of each of these statements is proved when $G$
is metrizable.

To study  approximable $\wap(G)$-interpolation sets we shall work within
the frame of
   locally compact $E$-groups. This is a class of groups
(introduced in \cite{C2}, see Section 4 for the definitions of $E$-groups and $E$-sets)
larger than that of locally compact $SIN$-groups
whose
members always admit  a good supply of weakly almost periodic
functions. The key concept here is that  of translation-compact sets,
also defined in Section 4. The main result  proved in this regard
is  the topologized analogue of a theorem obtained by Ruppert for
discrete groups (also valid for some semigroups) in \cite[Theorem
7]{rup}, parts of which  were also proven independently by Chou in
\cite{C4}.

We show first that in any locally compact group $G,$ if $T$ is a right (or left)
uniformly discrete,  approximable $\wap(G)$-interpolation set,
then $VT$ is translation-compact for some neighbourhood $V$ of the
identity $e$.

The converse is also true when $G$  is
a locally compact $SIN$-group or, more generally, when $G$ is a
 locally compact $E$-group and $T$ is an $E$-set.

We prove that in any locally compact $E$-group $G,$ any  right (or left) uniformly
discrete $E$-set $T$ such that $VT$
is translation-compact,  for some  neighbourhood $V$ of the identity $e$,    is
an approximable
$\wap(G)$-interpolation set.

We deduce that when $G$ is in addition metrizable, right (and left) uniformly
discrete sets such that $VT$ is translation-compact determine
completely the approximable $\wap(G)$-interpolation sets.

Some  consequences are obtained. We see first that
$B(G)$-interpolation sets in a metrizable locally compact Abelian
group (i.e., topological Sidon sets) are necessarily uniformly
discrete with respect to some neighbourhood $V$ of $e$  such that
$VT$ is translation-compact. In particular, Sidon sets in a discrete
Abelian group are translation-finite.

A second remarkable corollary follows. We consider the space of functions $f$ in $\wap(G)$ such $|f|$ is almost convergent to zero; that is \[\wap_0(G)=\{f\in \wap(G): \mu(|f|)=0)\},\] where
$\mu$ is the unique invariant mean on  $\wap(G)$. With an additional help of Ramsey theory, we deduce that
approximable $\wap(G)$- and
approximable $\wap_0(G)$-interpolation sets are in fact the same. It follows, as
a consequence of this characterization, that
 no infinite subset of
a metrizable locally compact group, in particular no infinite subset
of a discrete group, can be an approximable $\ap(G)$-interpolation
set.

Other interesting consequences follow. As in \cite{rup}, approximable
$\wap(G)$-interpo\-lation sets are also characterized in term of
their closure in the $\wap$-compactification $G^\wap$ of $G$. We
note that, as a consequence of these results,  sets of this sort
provide a combinatorial characterization of the points (called
strongly primes)  which do not belong to the closure of $G^*G^*$  in
the $\wap$-compactification $G^\wap$ of $G$,
 where $G^*$ is the remainder $G^\wap\setminus G$ .The same
  result is also true for the $\luc$-compactification of $G$.
 Some of these facts can already be seen in Theorem \ref{wap}, but
since they go out of the scope of the present paper, we wish to develop them further in a forthcoming paper \cite{FG2}.

 Section 5 is devoted to analyze the behaviour of approximable
 $\A(G)$-interpolation  sets under finite unions. We
prove that finite unions of
   approximable  $\A(G)$-interpola\-tion
  sets are approximable $\A(G)$-interpolation sets, provided the union is uniformly discrete.
    We deduce a union theorem for a class of
      translation-compact sets (precisely those obtained from
approximable      $\wap(G)$-interpolation sets). When $G$ is
discrete, this was proved by Ruppert in \cite{rup} after he
characterized right translation-finite sets. A direct combinatorial argument
proving this fact was provided recently in
\cite[Lemma 5.1]{FPr}.

We have included in  Section 6  some examples and  remarks along
with  some questions that are left open in the paper. We, for
instance, see how
 our characterizations of  $\A(G)$-interpolation sets fail
 in the absence of metrizability. We give an example (namely the Bohr compactification $\Z^\ap$) of a
compact non-metrizable group with an  $\ap(G)$-interpolation set
which  is neither left nor right uniformly discrete.
 Under (CH) such an example can be found in any nonmetrizable locally compact
group. We also see, by means of a simple
 example,  why the passage from the discrete to the locally compact case
needs some care:  a subset $T\subseteq  \R$ can be both an
approximable $\wap(\R_d)$-interpolation set and an approximable
$\luc(\R)$-interpolation set, without being an approximable
$\wap(\R)$-interpolation set.

\section{Preliminaries}
We recall now the definitions of our function algebras. We follow as
much as possible notation and terminology from  \cite{BJM} to which
the reader is directed  for more details. Let $\CB(G)$ be the
algebra of continuous, bounded, scalar-valued functions equipped with
its supremum norm and $\cnaught$ be the algebra of continuous
functions vanishing at infinity on $G.$ For each function $f$
defined on $G$, the left translate $f_s$ of $f$ by $s\in G$ is
defined on $G$ by $f_s(t)=f(st)$. A bounded function $f$ on $G$ is
{\it right uniformly continuous} when, for every $\epsilon>0$, there
exists a neighbourhood $U$ of $e$ such that
\[
|f(s)-f(t)|<\epsilon\quad\text{whenever}\quad st^{-1}\in U.\] These
are functions which are left norm continuous, i.e., \[s\mapsto f_s:
G\to \CB(G)\] is continuous,
and so in the literature, these functions are denoted also by
$\lc(G)$ or  $\luc(G)$. In this note, we shall use the latter
notation. In a like manner,  we shall denote by $ \ruc(G)$ the
algebra of \emph{left uniformly continuous} functions on $G$ and by
$\uc(G)$ the algebra of \emph{right and left uniformly continuous}
functions on $G$, hence $\uc(G)=\luc(G)\cap \ruc(G)$.

Let $\wap(G)$ and $\ap(G)$ be, respectively,  the algebra of weakly almost periodic functions and the algebra of almost periodic functions on $G.$
Recall that a  function $f\in\CB(G)$ is {\it weakly almost periodic} when the set of all
its left (equivalently, right) translates form a relatively weakly
compact subset in $\CB(G).$
A function $f$ is {\it almost periodic} when the set of all its
left (equivalently, right) translates form a relatively norm compact
subset in $\CB(G).$

If $\mu$ is the unique invariant mean
on $\wap(G)$ (see \cite{BJM}, or \cite{Bu}),  we put
\[\wap_0(G)=\{f\in \wap(G):\mu(|f|)=0\}.\]

The {\it Fourier-Stieltjes algebra} $B(G)$ is the space of
coefficients of unitary representations of $G$. Equivalently, $B(G)$
is the linear span of the set of all continuous positive definite
functions on $G$. As the Fourier-Stieltjes algebra is not uniformly
closed, we will consider its uniform closure $\B(G)$. Hence we have the {\it Eberlein algebra}
$\B(G)=\overline{B(G)}^{\|\cdot\|_\infty}$.

Recall that
\begin{align*}\cnaught\oplus\ap(G)\subseteq\B(G)\subseteq \wap(G)&=\ap(G)\oplus \wap_0(G)\\&\subseteq
\luc(G)\cap\ruc(G)\subseteq \luc(G)\subseteq \CB(G)\end{align*} (see
\cite{C1} for the first inclusion and \cite{BJM} for the rest).

If $\A(G)\subseteq \ell_\infty(G)$ is a $C^\ast$-subalgebra
there is a canonical morphism  $\epsilon_\A \colon G\to G^\A$ of $G$
into the spectrum (non-zero multiplicative linear functionals)
$G^\A$ of $\A(G)$ that is given by evaluations:
\[\epsilon_\A(g)(f)=f(g),\; \text{for every}\;f\in \A(G)\; \text{and}\; g\in G.\] This map
is continuous if $\A(G)\subseteq \CB(G)$ and injective if $\A(G)$ separates
points. In this case we will omit the function $\epsilon$ and
identify $G$ as a subgroup of $G^\A$.

We say that the $C^*$-algebra $\A(G)$ is admissible when it satisfies
the following properties: $1\in \A(G)$, $\A(G)$ is left translation
invariant, i.e., $f_s\in\A(G)$ for every $f\in \A(G)$ and $s\in G,$ and
the function defined on $G$ by $xf(s)=x(f_s)$  is in $\A(G)$ for every
$x\in G^\A$ and $f\in\A(G).$ When $\A(G)$ is admissible, $G^\A$ becomes a
semigroup compactification of the topological group $G$. This means
that $G^\A$ is a compact semigroup having a dense homomorphic image
of $G$ such that the mappings
\[
x\mapsto xy\colon  G^\A\rightarrow  G^\A \,\,\text{ and } \,\,
x\mapsto \epsilon_\A(s)x\colon G^\A\rightarrow  G^\A
\]
are continuous for every $y\in G^\A$ and $s\in G$. The product in
$G^\A$ is given by
\[xy(f)=x(yf)\quad \text{for every}\quad x,y\in G^\A\;\text{and}\; f\in\A(G)\]
When $\A(G)=\luc(G),$ the semigroup compactification $G^{\luc}$ is usually
referred to as the $\luc$-compactifica\-tion. It is the
largest semigroup compactification in the sense that any other
semigroup compactification is the quotient of $G^{\luc}.$
 When  $G$ is discrete, $G^{\luc}$ and the Stone-\v Cech compactification $\beta G$ are the same.
The   semigroup compactification $G^\wap$ is referred to as the
$\wap$-compactification, and it is the largest semitopological
semigroup compactification. The embedding $\epsilon_{\luc}$ is a
homeomorphism onto its image, hence  $G$  may be identified with its image in $G^\luc$. The same is true for $G^\wap$
when $G$ is locally compact. The closure of a set $X$ in $G^{\A}$ is
denoted by $\overline X^{\A},$ while in $G$ the closure of $X$ is
denoted as usual by $\overline X.$

A recent account on semigroup compactifications is given in \cite{gali10}.

\section{Approximable interpolation sets}
We start with the  main definition of the paper. We then identify
the $\A(G)$-interpolation sets for metrizable topological groups when
$\A(G)=C_0(G)$ and $\CB(G)$.  They are given, respectively,  by the
finite sets and the closed discrete sets.

\begin{definition} \label{def:main} Let $G$ be a topological group with identity $e$ and let
$\A(G)\subseteq \ell_\infty(G)$. A subset $T\subseteq G$ is said to be
\begin{enumerate}
  \item an \emph{$\A(G)$-interpolation set} if every bounded function
$f\colon T\to \C$ can be extended to a function $\tilde{f}\colon
G\to \C$ such that $\tilde{f}\in \A(G)$.
\item an \emph{approximable $\A(G)$-interpolation set} if it is an
 $\A(G)$-interpolation set and for every
 neighbourhood $U$ of $e$, there are open neighbourhoods $
V_1,V_2$ of $e$ with $\overline{V_1}\subseteq V_2\subseteq U$ such
that, for each $T_1\subseteq T$  there is   $h\in \A(G)$ with
$h(V_1T_1)=\{1\}$ and $h(G\setminus (V_2T_1))=\{0\}$.
\end{enumerate}
\end{definition}

The following is true in more general situations and in particular for any locally compact metrizable group. But since to
prove this fact we need the machinery developed later in next
section,  we content ourselves in the present section with the
following easy particular case.

\begin{proposition}\label{ex:noap}
  A  discrete, divisible  Abelian group $G$  does not have nontrivial approximable
$   \ap(G)$-interpolation sets.
\end{proposition}

\begin{proof}
Let $G$ be a discrete, divisible  Abelian group. It is well-known
that for Abelian groups the Bohr compactification of $G$ can be
identified with the group of \emph{all} characters of the character
group of $G$, i.e., $G^\ap=\left(\widehat{G}_d\right)^\wedge$,
\cite[Theorem 26.12]{hewiross1}. Since $G$ is divisible,
$\widehat{G}$  is torsion-free and duals of torsion-free Abelian
groups are connected \cite[Theorem 24.23]{hewiross1}, so $G^\ap$ is
connected.

Now it only remains to realize that if $T\subseteq G $ is an
approximable $\ap(G)$-interpolation set, its characteristic function
$1_{T}$ is almost periodic and therefore  $\overline{T}^{\ap}$ is
closed and open in $G^\ap$.
\end{proof}

\begin{proposition}\label{cb} Let $G$ be a topological group with identity $e.$
\begin{enumerate}
\item If $\A(G)\subseteq \CB(G)$, then the $\A(G)$-interpolation sets are discrete.
\item For the algebras $\CB(G)$ and $\luc(G),$ the interpolation sets are approximable interpolation sets.
\item If the underlying topological space of $G$ is normal, then the  discrete closed subsets of   $G$ are approximable $\CB(G)$-interpolation sets.
\item If $G$ is locally compact, then  the finite subsets of $G$ are approximable $\cnaught$-interpolation sets.
\item If $G$ is metrizable, then  the discrete closed sets are \emph{the} approximable $\CB(G)$-interpolation sets.
 \item If $G$ is locally compact and  metrizable, then  the finite subsets of $G$  are \emph{the} approximable $C_0(G)$-interpolation sets.
\end{enumerate}
 \end{proposition}

\begin{proof} Statements (i) and (iv) are clear.

To see Statement (ii), let $T\subseteq G$, $U$ be any neighbourhood of $e$ and choose a neighbourhood $V$ of $e$ such that
$V^2\subseteq U$. Then $V(V T)\subseteq UT$ , and so  $V(V T)$ and $(G\setminus (UT))$ are disjoint in $G$.
Therefore there exists $h\in \luc(G)$ such that
$h(\overline {V T})=\{0\}$ and $h((G\setminus (UT)))=\{1\}$.
Then the
statement clearly follows.

Urysohn's lemma together with Statement (ii) leads  immediately to
Statement (iii).

As for Statement (v), suppose otherwise that $T$ is discrete but not
closed, then pick a convergent sequence $(t_n)$ in $T$ with its
limit outside of $T$, and observe that the function $f$ defined on
$T$ such that $f(t_{n})=(-1)^n$ cannot be extended to a function in
$\CB(G)$.

For the last statement, suppose that $T$ is an approximable
$C_0(G)$-interpolation set. If $T$ is not finite, then since $T$
(being non-compact and closed by the previous statements) cannot
be contained in any compact subset of $G$, therefore a non-zero
constant function on $T$ do not extend to a  function in $C_0(G)$.
\end{proof}

\begin{remark}
In the non-metrizable situation, non-closed $\CB(G)$-interpolation
sets exist, see  Example \ref{ex:lucint} and Theorem \ref{CH}.
\end{remark}

In \cite[Corollary 2.2, page 49]{dunklrami}, it was  proved that $T$
is a Sidon set in a discrete Abelian group $G$ (i.e.,
$B(G)$-interpolation set) if and only if it is a
$\B(G)$-interpolation set. This is actually true for any discrete
group, as can be deduced from a result due to Chou (\cite[Lemma
3.11]{C4}). Chou's proof used direct functional analysis arguments
and does not rely on harmonic analysis tools.

As noted earlier $1_T\notin B(G)$ when $T$ is infinite. However,
$1_T$ may be a member of $\B(G).$ If a Sidon set $T$ satisfies this
property, then Dunkl and Ramirez called it a uniformly approximable
Sidon set (\cite[ Definition 5.3, page 59]{dunklrami}). Indeed,
Drury proved in \cite{drury70} that  Sidon sets in a discrete
Abelian group are uniformly approximable Sidon sets,  see also
 \cite[Theorem 5.7, page 61]{dunklrami}.
So with our terminology, each Sidon subset of a discrete Abelian
group $G$ is an approximable $\B(G)$-interpolation sets. We
summarize these observations in the following Proposition.

\begin{proposition}\label{openpb} Let $G$ be a discrete Abelian group and $T$ a subset of $G.$ Then the following statements are equivalent.
\begin{enumerate}
\item $T$ is a $B(G)$-interpolation set.
\item $T$ is a $\B(G)$-interpolation set.
\item $T$ is an approximable  $\B(G)$-interpolation set.
\end{enumerate}
\end{proposition}

When $G$ is a  metrizable locally compact Abelian group, the
implication (i) $\Longrightarrow$ (iii) of the previous proposition
is valid.
 For this, we need the following lemma due to  Dechamps-Gondim (\cite[Theor\`eme 1.1 and Theor\`eme
 2.1]{MDG}). We note that $B(G)$-interpolation sets are called
 \emph{topological Sidon} sets in \cite{MDG}.

\begin{lemma}[Dechamps-Gondin] \label{Myriam} Let $G$ be a metrizable locally compact Abelian group and $T$ be a $B(G)$-interpolation set.
Then for every $\beta>0$ and every neighbourhood $U$ of the
identity, there exist a constant $M$ and a compact subset $K\subseteq
\widehat{G}$ such that for every finite subset $X$ of $T$ there
exists a function $f\in L_1(\widehat{G})$ with  support contained in
$K$ and \begin{enumerate}
\item $\|f\|_1\leq M$,
\item $\widehat{f}_{|X}=1$,
\item $|\widehat{f}(g)|<\beta$ for every $g\notin UX.$
\end{enumerate}
\end{lemma}

%

\begin{proposition}\label{prop:bg} Let $G$ be a metrizable locally compact Abelian group.
Then the $B(G)$-interpolation sets are approximable
$\B(G)$-interpolation sets.
\end{proposition}

\begin{proof}
Let $T$ be a $B(G)$-interpolation set. Choose a  neighbourhood $U$
of the identity and  $\beta<\frac12$. Take from Lemma \ref{Myriam}
the corresponding $K\subseteq \widehat{G}$ and $M>0$.

We first observe that the Fourier-Stieltjes transforms
$\widehat{\mu}$ of measures  $\mu\in M(\widehat{G})$ with  support
contained in  $K$ and $\|\mu\|\leq M$ all  have a common modulus of
uniform continuity: indeed if $V_{K,M}=\{g\in G\colon
|\chi(g)-1|<\varepsilon/2M, \mbox{ for all } \chi \in K\}$, then
$V_{K,M}$ is a neighbourhood of $e$ and $gh^{-1}\in V_{K,M}$ implies
that $|\widehat{\mu}(g)-\widehat{\mu}(h)|<\varepsilon$.
 Let $V_1\subseteq V_{K,M}$ be a neighbourhood of $e$ such that $\overline{V_1}\subseteq U$. We now see that $V_1$ and $U$ suffice to show that $T$ is an approximable
 $\B(G)$-interpolation set. Let to that end $T_1\subseteq T$.  By Lemma
\ref{Myriam}, for each finite subset $X\subseteq T_1$, there is
$f_X\in L_1(\widehat{G})$  with  $\|f_X\|_1\leq M$,
$\widehat{f_X}(X)=\{1\}$ and $|\widehat{f_X}(g)|<\beta$ for every
$g\in G\setminus UX$. Consider the net $(f_X)_X$ where $X$ runs over
all finite subsets of $T_1$ ordered by inclusion and let
 $\mu\in M(\widehat{G})$ be  the limit of some subnet of  $(f_X)_X$ in the weak $\sigma(M(G),C_0(G))$-topology (recall that the ball of radius $M$ in $M(G)$ is compact in this topology). Then $\|\mu\|\leq M$, $\widehat{\mu}(T_1)=\{1\}$ and
$|\widehat{\mu}(g)|<\beta$ for every $g\in G\setminus UT_1$. Let
$\psi=\widehat{\mu}\in B(G)$ for the remainder of this proof.

Extend $\psi$ to a continuous function $\widetilde{ \psi}$ on the
Eberlein compactification $G^{\B}$ of $G$. Take  $0<\epsilon
<1-\beta$ and  choose  a continuous function $\rho:\mathbb R\to \mathbb
R$ such that $\rho((-\beta,\beta))=\{0\}$ and
$\rho((1-\epsilon,1+\epsilon))=\{1\}$. Define finally   $\tilde
\phi=\rho\circ\widetilde{\psi}$. Then $\tilde\phi\in \CB(G^{\B}),$ and
so $\phi:=\tilde\phi\restr_G\in \B(G)$. Moreover, if $s\in V_1T_1$
then $st^{-1}\in V_1$ for some $t\in T_1$, and so
\[
|\psi(s)-1|=|\psi(s)-\psi(t)|<\epsilon,\] showing that $\psi(s)\in
(1-\epsilon,1+\epsilon)$, and therefore,
$\phi(s)=\rho(\psi(s))=1$. If, on the other hand,  $s\notin UT_1$, then
$\phi(s)=\rho(\psi(s))=0$, since $ \psi(s)\in (-\beta,\beta).$ We have
thus that $\psi(V_1T_1)=\{1\}$ and $\psi(G\setminus UT_1)=\{0\}$.
  This completes the proof.
\end{proof}

\section{Characterizing approximable $\luc(G)$- and $\wap(G)$-interpolation sets}

 The case of $\luc(G)$ is not as straightforward as that of $\cnaught$ or $\CB(G)$. But we  prove in Theorem
 \ref{lem:lucintconv}  that the $\luc(G)$-interpolation sets are
 given by the right uniformly discrete sets when $G$ is metrizable.

The situation becomes more delicate with the algebra $\wap(G)$.
Here, if $G$ is   minimally weakly almost periodic (as, for
instance, the group $SL(2,\mathbb R)$) then only finite sets are
$\wap(G)$-interpolation sets, for in this case
$\wap(G)=\cnaught\oplus\mathbb C1$. So we shall work with  groups
which have a good supply of non-trivial weakly almost periodic
functions such as $SIN$-groups, or
more generally, $E$-groups  (see \cite{BF} or \cite{F07}).

Theorem \ref{wap} concerns mainly
approximable $\wap(G)$-interpolation sets with $G$  a locally
compact $E$-group. This theorem
 is the topologized version of a
theorem of  Ruppert  proved in \cite{rup} for some infinite discrete
semigroups. We identify the approximable $\wap(G)$-interpolation
sets when $G$ is metrizable. They are given by the  left (or right)
uniformly discrete sets $T$ such that $VT$ is translation-compact for some neighbourhood $V$ of the identity.

Another interesting
consequence is given in Theorem \ref{wap=wap_0} which shows that
left (or right) uniformly discrete, approximable
$\wap(G)$-interpolation sets  and left (or right) uniformly discrete,
approximable $\wap_0(G)$-interpolations sets are in fact the same sets $T$, they
both have the property that $VT$  is translation-compact for some neighbourhood $V$ of the identity.
When $G$ is metrizable,
the uniform discreteness  turns to be a
necessary condition as well.

We should point out that part of Theorem \ref{wap} (and so also part of Theorem \ref{wap=wap_0})  is proved beforehand in Corollary \ref{cor:wap2} for any locally compact group.
Theorem \ref{wap=wap_0} together with Corollary \ref{cor:wap2}  implies, as already promised before Proposition
\ref{ex:noap}, that uniformly discrete sets can never be
approximable $\ap(G)$-interpolation sets in this class of groups.

\begin{definition}
Let $G$ be a (non-compact) topological group. We say that a subset
$T$ of $G$ is
\begin{enumerate}
\item (Ruppert, \cite{rup})
{\it right translation-finite} if every  infinite subset $L\subseteq
G$ contains a finite subset $F$ such that $\bigcap\{b^{-1}T\colon
b\in F\}$ is finite; {\it left translation-finite} if every infinite
subset $L\subseteq G$ contains a finite subset $F$ such that
$\bigcap\{Tb^{-1}\colon b\in F\}$ is finite; and {\it
translation-finite} when it is both right and left
translation-finite.
\item
{\it right translation-compact} if every  non-relatively compact
subset $L\subseteq G$ contains a finite subset $F$ such that
$\bigcap\{b^{-1}T\colon b\in F\}$ is relatively compact; {\it left
translation-compact} if every  non-relatively compact subset
$L\subseteq G$ contains a finite subset $F$ such that
$\bigcap\{Tb^{-1}\colon b\in F\}$ is relatively compact; and
\emph{translation-compact} when it is both left and right
translation-compact.
\item a \emph{right $t$-set} (\emph{left $t$-set}) if there exists a compact subset $K$ of $G$ containing $e$ such that $gT\cap T$ (respectively, $Tg\cap T$) is
relatively compact for every $g\notin K$; and a \emph{$t$-set} when
it is both a  right and a left $t$-set.
 \end{enumerate}
\end{definition}

Right $t$-sets  were used originally by Rudin in  \cite{ru} to
construct  weakly almost
  periodic functions which
   are not in $\B(G)$ for non-compact locally
 compact Abelian groups with a closed discrete
  subgroup of unbounded order. Then they were used by
   Ramirez in \cite{ra} for the same purpose when
   $G$ is a non-compact locally compact Abelian group.
    They came up again in the non-Abelian situation
     when Chou proved in \cite{C1} the same result
     if $G$ is nilpotent or if $G$ is an $IN$-group
     (i.e., $G$ has an invariant compact neighbourhood
     of the identity).
In another paper \cite{C2}, Chou used these sets to construct
funtions in $\wap_0(G)$ which are not $\cnaught$.

 More recently, the $\wap$-functions defined with the help of
 right $t$-sets enabled Baker and Filali \cite{BF} and Filali \cite{F07} to study some algebraic properties of the $\wap$-compactification $G^\wap$ of $G.$

Ruppert  \cite[Theorem 7]{rup} and Chou \cite[Proposition 2.4]{C4},
proved that the translation-finite subsets of discrete groups
(called $R_W$-sets in \cite{C4}) precisely coincide with the
approximable $\wap(G)$-interpolation sets.
 The same class of sets was used by Filali and Protasov in \cite{FPr} to
 characterize the strongly prime ultrafilters in the Stone-\v Cech
 compactification  $\beta G$ of a discrete group $G$.
They were called sparse sets.

It is evident that the right (left) $t$-sets are also right (left)
translation-compact. Examples of right translation-finite sets which
are not right $t$-sets are easy to construct (see also \cite[Examples
11]{rup}), and even examples of right translation-finite sets which are
not finite unions of right $t$-sets were  devised in the discrete case
by Chou  \cite[Section 3]{C4}. $t$-Sets having compact covering as
large as that of $G$ may be constructed by induction in any
non-compact locally compact group $G$, see the following example
taken from \cite{F07}. Recall that, for a topological space  $X$,
the compact covering  of $X$ is   the minimal number  $\kappa(X)$ of
compact sets required to cover $X$.

\begin{example}\label{$t$-set}
Let $G$ be a non-compact locally compact group and fix a compact
symmetric neighbourhood $V$ of the identity $e$ of $G$. Start with
$t_0=e,$ say. Let $\alpha<\kappa(G)$ and suppose that the elements
$t_\beta$ have been selected for every $\beta<\alpha$. Set \[T_\alpha=
\bigcup\limits_{\beta_1,\beta_2,\beta_3<\alpha}V^2t_{\beta_1}^{\epsilon_1}t_{\beta_2}^{\epsilon_2}V^2t_{\beta_3}^{\epsilon_3},\]
where each $\epsilon_i=\pm1.$ Then $\kappa(T_\alpha)<\kappa(G),$ and
so we may select an element $t_\alpha$ in $G\setminus T_\alpha$
for our set $T$. In this way, we form a set
$T=\{x_\alpha:\alpha<\kappa(G)\}$. As already checked in \cite{F07},
the set $T$ is right $V^2$-uniformly discrete (the definition is
given below), has $\kappa(T)=|T|=\kappa(G)$, $s(VT) \cap (VT)$ and $(VT)s \cap (VT)$
are relatively compact for every $s \not\in V^2$, i.e., $VT$ is a
$t$-set.
\end{example}

We proceed now to prove the analogues of Statements (iii)-(vi) of
Proposition \ref{cb} for the algebras $\luc(G)$ and $\wap(G)$. We
shall need right (left) uniformly discrete sets  and
translation-compact sets instead of discrete closed sets or finite
sets.
  The relevance of translation-compact sets in our setting is made clear in  Lemma \ref{lem:tc} below. We also see in Example \ref{ex:tc} that translation-finite sets do not suffice to characterize approximable $\wap$-interpolation sets.

A crucial  tool for the rest of the paper is Grothendieck's
criterion which we recall now.
For the proof, see for example \cite[Theorem 4.2.3]{BJM}, or the
original paper of Grothendieck \cite{gr}.

{\it A bounded function $ f: G \to\C$ is weakly almost periodic if and
only if for every pair of sequences $(s_n)$ and $(t_m)$ in $G$,
\[\lim_n\lim_m f(s_nt_m)=\lim_m\lim_nf(s_nt_m)\]
whenever the limits exist.}

The criterion may be equivalent stated in
terms of ultrafilters: a bounded function $ f: G \to\C$ is weakly
almost periodic if and only if for every pair of sequences $(s_n)$
and $(t_m)$ in $G$ and
 every pair of free ultrafilters, $\mathcal{U}$ and $\mathcal{V}
$ over $\N$, the limits of the function along the ultrafilters coincide, i.e.,
\[\lim_{n,\mathcal{U}}\lim_{m,\mathcal{V}} f(s_nt_m)=
\lim_{m,\mathcal{V}}\lim_{n,\mathcal{U}}f(s_nt_m).  \]

The proof  of the following Lemma extracts the basic idea of Lemma B
of \cite{BF}.
 \begin{lemma}\label{lem:tc}
   Let $G$ be a topological group and let $T\subseteq G$. If
    $T$ is  translation-compact,
 then every right and left uniformly
    continuous function   supported in $T$ is weakly almost periodic
    and its extension to $G^\wap$ vanishes on $G^{\ast}G^\ast$.
   \end{lemma}

   \begin{proof}
Let $f\colon G\to \C$ be in $\uc(G)$ and supported in $T$.
We first prove that whenever $(s_\alpha) $ and $(t_\beta)$ are nets in $G$ that do not accumulate in $G$, then
\[\lim_\alpha \lim_\beta f(s_\alpha t_\beta)=\lim_\beta\lim_\alpha f(s_\alpha
t_\beta)=0.\]
Suppose towards a contradiction  that $\lim_\alpha\lim_\beta f(s_\alpha t_\beta)\neq 0.$ This means that
there is $\alpha_0$ such that for every $\alpha\geq \alpha_0$, we
have $\lim_\beta f(s_\alpha t_\beta)\neq 0$. Since $(s_\alpha)$ does
not cluster in $G$  and $T$ is translation-compact, there must exist
a finite set $\{s_{\alpha_1},\ldots,s_{\alpha_k}\}$ with
$\alpha_i\geq \alpha_0$ for each $i=1,2,...,k$ such that $\bigcap_{i=1}^k
s_{\alpha_i}^{-1}T$ and $\bigcap_{i=1}^k Ts_{\alpha_i}^{-1}$ are
relatively compact. Since $\lim_\beta f(s_\alpha t_\beta)\neq 0$,
there must be an index $\beta_0$ such that
$f(s_{\alpha_i}t_\beta)\neq 0$ for any $\beta\geq \beta_0$ and
$i=1,2,\ldots , k$. We deduce thus that
\[ t_\beta\in \bigcap_{i=1}^k s_{\alpha_i}^{-1}T
\mbox{ for    all } \beta\geq \beta_0 .\] Since the latter  set  is relatively
compact, we deduce that $(t_\beta)$ has a cluster point in $G$, a
contradiction showing that $\lim_\alpha\lim_\beta f(s_\alpha
t_\beta)= 0$.
 Arguing symmetrically,  we can prove that
$\lim_\beta\lim_\alpha f(s_\alpha t_\beta)= 0$.

We now prove that $f\in \wap(G)$.
Take two sequences  sequences $(s_n)$ and $(t_m)$ in $G$
such that both limits $\lim_n\lim_m f(s_nt_m)$ and $\lim_m\lim_n f(s_nt_m)$
exist. If neither $(s_n)$ nor $(t_m)$ accumulate in $G$, the above argument shows that \[
\lim_n\lim_m f(s_nt_m)=\lim_m\lim_n f(s_n
t_m)=0.\]
If  $(s_n)$ has a  cluster point $g\in G$, we choose
 a cluster point  $q\in G^\luc$ of $(t_m)$, and taking into account that  the multiplication  is jointly
continuous on $G\times G^\luc$ and passing to subnets if necessary, we  see that \[\lim_n\lim_m
f(s_n t_m)=f^\luc(gq)=\lim_m \lim_n f(s_nt_m),\] where $f^\luc$ is the extension of $f$ to $G^\luc$.
If, alternatively, it is $(t_m)$ that accumulates at some point of $G$,  we argue in the
same way using $G^\ruc $ instead of $G^\luc$ (the function $f$ is
assumed to be in $\luc(G)\cap \ruc(G)$).

We obtain, that  in any case,
\[\lim_n \lim_m f(s_n t_m)=\lim_m\lim_n f(s_n t_m),\] and,  as a consequence of Grothendiecks's criterion, $f
\in \wap(G)$.

Having proved that $f\in \wap(G)$, the last assertion of the theorem
 is a straightforward consequence of the argument in the first paragraph of this proof. If $p,q\in G^\ast G^\ast$, then $p=\lim_\alpha s_\alpha$,
$q=\lim_\beta t_\beta$ for some nets $(s_\alpha)\in G$ and
$(t_\beta)\in G$. Since the nets $(s_\alpha)$ and $(t_\beta)$ do not
accumulate in $G$, the mentioned argument shows that $f^\wap(pq)=0$.
   \end{proof}

 \begin{definition}
   Let $G$ be a topological group with identity $e$,  $T$ be a  subset of $G$ and $U$ be a neighbourhood of $e$.
We say that $T$ is
\begin{enumerate}
\item
 \emph{right $U$-uniformly
discrete}  if $Us\cap Us^\prime=\emptyset$ for every $s\neq
s^\prime\in T$.
\item \emph{left $U$-uniformly discrete}   if
 $sU\cap s^\prime U=\emptyset$  for every $s\neq s^\prime\in T$.
\item \emph{right uniformly discrete} (respectively, \emph{left
uniformly discrete}) when it is right $V$-uniformly discrete
(respectively, left $V$-uniformly discrete) with respect to some
neighbourhood $V$ of $e$.
\item uniformly discrete if it is both left and right uniformly  discrete.
\end{enumerate}
\end{definition}

The functions we introduce below constitute an important tool to
reflect combinatorial properties at the function algebra level. When
$f$ is the constant 1-function, they provide pointwise
approximations to the characteristic function.

\begin{definition}
 Let  $T$ be a subset of a topological group $G$ with identity $e$, and suppose that $T$ is right (respectively,  left) $U$-uniformly discrete with respect to some neighbourhood $U$ of $e$.  Let $V$ be a symmetric neighbourhood of $e$ such that
 $V^2\subseteq U.$
For each  $\psi \in \luc(G)$ (respectively, $\psi\in\ruc(G)$)  with
$\psi(e)=1$ and $\psi(G\setminus V)=\{0\}$, and for each bounded
function $f\colon T\to \C$, we define a function $f_{T,\psi}\;
(\text{respectively,}\; f_{\psi,T} )\colon G\to \C$ by
\[ f_{T,\psi}(s)=\sum_{t\in
T}f(t)\psi\left(st^{-1}\right)\quad \left(\text{ respectively,}\;
f_{\psi,T}(s)=\sum_{t\in
T}f(t)\psi\left(t^{-1}s\right)\right).\qquad(\ast)\]
\end{definition}

Let $G$ be non-compact and recall from \cite{C1}, that an {\it
E-set} is a non-relatively compact subset $T$ of $G$ such that for
each neighbourhood $U$ of $e,$ the set \[\bigcap \{t^{-1}Ut: t\in
T\cup T^{-1}\}\] is again a neighbourhood of $e.$ When such a set
exists in $G$, we say that $G$ is an {\it E-group}. It is clear that
non-compact $SIN-$groups are $E$-groups (with every non-relatively
compact subset being an $E$-set). Direct products of any
$E$-group with any topological group are again $E$-groups.
The groups with a non-compact centre such as the matrix group
$GL(n,\mathbb{ R})$ belong also to this class.

 \begin{lemma}\label{lem:tcc} Let   $G$ be  a topological group with identity $e$,
 and let $U$,  $V$ be  symmetric neighbourhoods of $e$ such that
 $V^2\subseteq U$ and let $T$ be right $U$-uniformly discrete
  (respectively, left $U$-uniformly discrete). Consider    $\psi\in \luc(G) $ (respectively, $\psi\in \ruc(G) $) with
 $\psi(e)=1$ and $\psi(G\setminus V)=\{0\}$. Then
\begin{enumerate}
\item $f_{T,\psi}(G\setminus VT)=\{0\} $ (respectively, $f_{\psi,T}(G\setminus TV)=\{0\}$).
  \item $f_{T,\psi}\in\luc(G)$ (respectively, $ f_{\psi,T}\in \ruc(G)$) and
\[ f_{T,\psi}(vt)= f_{\psi,T}(tv)=f(t)\psi(v) \quad \mbox{for all $v\in V$ and $t\in T$}.\]
\item If $G$ is an E-group, $T$ is  an $E$-set and $\psi\in\uc(G)$, then $f_{T,\psi}\in \uc(G)$ (respectively, $f_{\psi,T}\in \uc(G)$).
If, in addition, $VT$ (respectively, $TV$) is translation-compact, then
$f_{T,\psi}\in \wap(G)$ (respectively, $f_{\psi,T}\in \wap(G)$).
\end{enumerate}
\end{lemma}

\begin{proof} Statement
(i) follows immediately from the definition of $f_{T,\psi}$.
  The proof of (ii) is precisely  \cite[Exercise 4.4.16]{BJM}).
 If $g,h\in G$ are
such that $gh^{-1}\in V$, and $g\in Vt$ for some $t\in T$, then
$h\in V^2t$, while $h\notin Vt^\prime$ for any $t\neq t^\prime \in
T$. So for arbitrary $g,h\in G$ with $gh^{-1} \in V$, either
$f_{T,\psi}(g)=f_{T,\psi}(h)=0$ or there is $t\in T$ with
\[
\left|f_{T,\psi}(g)-f_{T,\psi}(h)\right|=\left|f(t)\psi(gt^{-1})-f(t)\psi(ht^{-1})\right|.\]
As $\psi \in \luc(G)$, we see that $f_{T,\psi}$ is   right uniformly
continuous, i.e., $f_{T,\psi}\in\luc(G)$.

To prove (iii), we check that $f_{T,\psi}$ is left uniformly
continuous as well. The case of $f=0$ is trivial, so suppose that
$f\ne0.$ Given $\epsilon>0,$ choose a neighbourhood $V_0$ of the
identity in $G$ with $V_0 \subseteq V$ and
\[|\psi(u)-\psi(v)|<\frac{\epsilon}{\|f\|}\quad\text{ whenever}\quad u^{-1}v\in
V_0.\] Let $W$ be a neighbourhood of the identity such that $t W
t^{-1}\subseteq V_0$ for every $t\in T.$

If $g\in VT,$ then $g=ut$ for some $u\in V$ and $t\in T,$ and so
$g^{-1}h\in W$ implies that \[h\in gW=utW\subseteq uV_0t\subseteq
V^2t,\] i.e., $h=vt$ for some $v\in U$. Now from
$g^{-1}h=t^{-1}u^{-1}vt\in W$, we conclude that $u^{-1}v\in
tWt^{-1}\subseteq V_0.$ Thus  \[|f_{T,\psi}(g) - f_{T,\psi}(h)| =
|f(t)\psi(u)-f(t)\psi(v)| \le\|f\||\psi(u) - \psi(v)|<\epsilon.\]
The argument is similar when $h\in VT,$ and it is  trivial when
neither $g$ nor $h$ is in $VT$. Since $f_{T,\psi}$ is both left and
right uniformly continuous and is supported in $VT$, the last part
of  Statement (iii) follows now from Lemma \ref{lem:tc}.
\end{proof}

 \begin{remark} \label{remrd}
Functions $\psi\in \uc(G)$ with $\psi(e)=1$ and $\psi(G\setminus V)=\{0\}$ for a given neighbourhood $V$ of the identity are always available since,
by \cite[Definition-proposition 2.5]{roeldier} the infimum of the left and right uniformities on $G$ (the so-called lower uniformity or Roelcke uniformity) induces the original topology on $G$.
\end{remark}

    \begin{lemma}\label{lem:lucint}
Let $G$ be a topological group with identity $e$ and let $T\subseteq G$.
\begin{enumerate}
\item If $T$ is  right uniformly discrete (respectively, left uniformly discrete),
then  $T$ is an approximable $\luc(G)$-interpolation set (respectively,
$\ruc(G)$-interpolation set).
\item If $G$ is an $E$-group, and $T$ is an $E$-set that is  right (respectively, left) $U$-uniformly discrete with respect some neighbourhood $U$ of the identity
and $VT$  (respectively, $TV$) is translation-compact for some neighbourhood $V$ of the identity with $V^2\subseteq U$, then $T$ is an
approximable $\wap(G)$-interpolation set.
\end{enumerate}
\end{lemma}

\begin{proof}
Suppose $T$ is right uniformly discrete and let $U$ be a
neighbourhood of the identity
 with $Us\cap Us^\prime=\emptyset$ for
 every $s\neq s^\prime \in T$.
By Proposition \ref{cb}, we only need to check that $T$ is an
$\luc(G)$-interpolation set. Let $V$ be another neighbourhood of the
identity with $V^2\subseteq U$ and let $\psi\in \luc(G)$ with support
contained in $V$ and $\psi(e)=1$. If $f\colon  T\to \C$ is any
bounded function, then   the  function $f_{T,\psi}$ defined above in
$(\ast)$ is an extension of $f$ and it is in $\luc(G)$ by the
previous lemma.

Suppose now that $G$ is an $E$-group. Choose a function $\psi \in\uc(G)$ with support
contained in $V$ and  $\psi(e)=1$ (see Remark \ref{remrd}). If Condition (ii) is satisfied,
the above scheme will again show that
    $T$ is a $\wap(G)$-interpolation set using  (iii) of Lemma \ref{lem:tcc}.
 To see that $T$ is an  approximable $\wap(G)$-interpolation, take  two neighbourhoods of the identity $V_1$ and  $V_2$ with
$   V_1^3\subseteq V_2\subseteq U.$   Then
$V_1\overline{V_1}\subseteq V_1 V_1^2 \subseteq V_2$. Therefore,
$\overline V_1$ and $G\setminus V_2$ are closed sets in $G$ with
$V_1\overline{ V_1}\cap (G\setminus V_2)=\emptyset.$ So we may pick
$\psi\in \luc(G)$ with $\psi(V_1)=1$ and $\psi(G\setminus V_2)=0$.
If now $T_1$ is any subset of $T$, then  the function $1_{T_1,\psi}$
as defined in ($\ast$) with the constant function $1$ on $T_1$ is in
$\wap(G)$ again by  Lemma \ref{lem:tcc}. Since the function clearly
satisfies  $1_{T_1,\psi}(vt)=1$  for every  $v\in V_1$ and $t\in
T_1$ and $1_{T_1,\psi}(G\setminus V_2T_1)=\{0\}$, the proof is
complete.
 \end{proof}

It turns out that  the converses to the statements in Lemma \ref{lem:lucint} are valid
when $G$ is metrizable. We begin with the case of $\luc(G)$.

\begin{theorem}\label{lem:lucintconv} Let  $G$ be a metrizable topological group and let $T$ be a subset
of $G$. The following assertions are equivalent.
\begin{enumerate}
\item $T$ is  an approximable $\luc(G)$- (respectively, $\ruc(G)$-)interpolation set.
\item $T$ is  an  $\luc(G)$- (respectively, $\ruc(G)$-)interpolation set.
\item $T$  is right (respectively, left) uniformly discrete.
\end{enumerate}
 \end{theorem}
\begin{proof}
Lemma \ref{lem:lucint} proves that (iii) implies (i) and that  (i)
implies (ii)  is  a matter  of definition. We only have to see that
(ii) implies (iii).

Suppose that  $T$ is  an  $\luc(G)$-interpolation set that  is not
right uniformly discrete.
 Consider a neighbourhood basis at the identity $(U_n)_n$  consisting of
symmetric neighbourhoods  such that $U_{n+1}^2 \subseteq U_n$. Recall
that $T$ is necessarily discrete.

Since  $T$ is not right uniformly discrete, we can find $s_1,t_1\in
T$, $s_1\neq t_1$, such that $s_1t_1^{-1}\in U_1^2$.

Suppose now that we have chosen $2n$ different points $\{t_1,\ldots,
t_n,s_1,\ldots,s_n\}$ and $1<k_1<k_2<\ldots k_n$ in such a way that
$s_nt_n^{-1}\in U_{k_n}^2$, but $s_n\neq t_n$.

Using  the fact that $T$ is discrete and that $t_i\neq s_j$ for all
$1\leq i,j\leq n$, we then choose $k_{n+1}>k_n$ such that:
\begin{enumerate}
 \item $\left(T\setminus \{t_i\}\right) \cap\left( U_{k_{n+1}}t_i\right)=\left(T\setminus \{s_i\}\right) \cap \left( U_{k_{n+1}}s_i\right)
  =\emptyset$ for all $i=1,\ldots,n$
  \item
  $ t_is_j^{-1}\notin U_{k_{n+1}}^2 \mbox{ for all $1\leq i,j\leq
  n$} $.
  \item $t_i t_j^{-1}\notin U_{k_{n+1}}^2, \;\;
s_i s_j^{-1}\notin U_{k_{n+1}}^2 \mbox{ for all $1\leq i,j\leq n$,
$i\neq j$}.$
  \end{enumerate}

Then \[\left(T\setminus\{t_1,\ldots, t_n,s_1,\ldots,s_n\}\right)
\left(T\setminus\{t_1,\ldots, t_n,s_1,\ldots,s_n\}\right)^{-1}\cap
U_{k_{n+1}}^2 \neq \{e\},\] for, otherwise, we would have
$TT^{-1}\cap U_{k_{n+1}}^2=\{e\}$ and this would imply that $T$ is
right uniformly discrete.

We now choose $t_{n+1}, s_{n+1}\in T\setminus\{t_1,\ldots,
t_n,s_1,\ldots,s_n\})$ with $t_{n+1}\neq s_{n+1}$ such that
$t_{n+1}s_{n+1}^{-1}\in U_{k_{n+1}}^2$.

We have constructed in this way two faithfully indexed sequences
\[T_1=\{t_n \colon n\in \N\}\subseteq T\quad\text{ and}\quad T_2=\{s_n
\colon n\in \N\}\subseteq T\] in such a way that $T_1\cap
T_2=\emptyset$ but $t_ns_n^{-1}$ converges to the identity.

Since $T$ is an $\luc(G)$-interpolation set, there is $f \in
\luc(G)$ so that $f(t_n)=1$ and $f(s_n)=-1$.
 By uniform continuity, there is
a neighbourhood $U_{0}$ of the identity such that $|f(x)-f(y)|<1$ if
$xy^{-1}\in U_{0}$. Taking  then $n$ such that $U_{k_n}^2\subseteq
U_0$, we reach a contradiction as $s_{n}t_{n}^{-1}\in U_{0}$, but
$f(s_n)-f(t_n)=2$.
\end{proof}

\begin{remark}
In the non-metrizable situation, non-uniformly discrete sets may be
$\luc(G)$-interpolation sets. In fact Example \ref{ex:lucint} gives
$\ap(G)$-interpolation sets which are not uniformly discrete.
Theorem \ref{CH} at the end of the paper suggests  that  metrizability is essential
for Theorem \ref{lem:lucintconv} to be true.
\end{remark}

We now deal with the $\wap(G)$-case.  To avoid making the proof of
the main theorem  too cumbersome, we begin by establishing the
following Lemma as well as Theorem \ref{inherit}.

\begin{lemma}\label{Ruppert} Let $G$ be a locally compact group with a symmetric relatively compact neighbourhood U of the identity and a right $U$-uniformly discrete set
$T$.
 If  $UT$   is not  translation-compact, then $T$ contains a subset $T_1$
  such that no $f\in \ell_\infty(G)$ with $f(T_1)=\{1\}$
  and $f(G\setminus UT_1)=\{0\}$ is  weakly almost periodic.
  \end{lemma}

  \begin{proof}
We adapt the argument used by Ruppert in \cite{rup} in the discrete
case. Suppose that $UT$ is not right translation-compact. Then there exists a
non-relatively compact subset $L$ of $G$ which contains no finite
subset $F$ for which $\bigcap_{b\in F}b^{-1}UT$ is relatively
compact. Define inductively two sequences $(s_n)$ and $(t_m)$ in $G$
as follows. Start with $s_1\in L$ and let $t_1\in s_1^{-1}T$.
Suppose that $s_1, s_2,...,s_n$ have been selected in $L$ and $t_1,t_2,...,t_n$ have been
selected in $G$. Then we may choose $s_n\in L$ such that
\begin{equation}\label{*}s_n\notin U^4\{s_kt_lt_m^{-1}:1\le k\le l<n,\quad 1\le
m<n\}.\end{equation} This is possible since the latter set is
relatively compact while $L$ is not. Take then \[t_n\in
\bigcap_{m\le n}s_m^{-1}UT,\] but
\begin{equation}\label{**}t_n\notin s_m^{-1}U^4\{s_lt_{k}:1\le k<
l\le n\}\quad\text{ for all}\quad m\le n.\end{equation} This is again
possible because the second set is relatively compact while
the first one is not. Note
 that, by choice,
we have $s_mt_n\in UT$ for every $m\le n.$ So for every $m\le n,$
there is a unique  $t_{mn}\in T$ such
 that $s_mt_n\in Ut_{mn}.$ Let
\[T_1=\{t_{mn}\in T:1\le m\le n<\infty\}.\]
 We claim first that
\begin{equation}\label{***}
U^2\{s_mt_n:m\le n\}\cap U^2 \{s_mt_n:m>n\}=\emptyset.\end{equation}
So let us consider two elements $us_\alpha t_\beta$ and
$vs_{\alpha'}t_{\beta'}$ with $\alpha\le \beta$, $\alpha'>\beta'$
and $u,$ $v\in U^2$. If $\beta\le \beta',$ put $\alpha=k,$
$\beta=l$, $\alpha'=n$ and $\beta'=m,$ then \[us_\alpha
t_\beta=vs_{\alpha'}t_{\beta'}\] implies that
\[us_kt_l=vs_nt_m\quad\text{ with}\quad 1\le k\le l\le m<n,\]
contradicting \eqref{*}. If $\beta>\beta',$ put $\alpha=m,$
$\beta=n$, $\alpha'=l$ and $\beta'=k,$ then  \[us_\alpha
t_\beta=vs_{\alpha'}t_{\beta'}\] implies that
\[us_mt_n=vs_lt_k\quad\text{ with}\quad 1\le k<
l\quad\text{and}\quad  m\le n.\] If $\alpha'\le\beta$, then $l\le n$
and so this clearly contradicts \eqref{**}. If $\alpha'>\beta,$ then
we obtain
\[vs_lt_k=us_mt_n\quad\text{ with}\quad 1\le m\le n< l\quad\text{and}\quad  1\le k< l,\]
contradicting \eqref{*}.

Accordingly,
\begin{equation} U\{t_{mn}:m\le n\}\cap
\{s_mt_n:m>n\}=\emptyset.\end{equation} For if $ut_{pq}=s_mt_n$  for
$u\in U$, $p\le q$ and $m>n$, then $uvs_pt_q=s_mt_n$ for $u\in U$,
$p\le q$, $m>n$ and some $v\in U$, which is not possible by
\eqref{***}.

Let now  $f$ be any bounded function on $G$ with
 $f(T_1)=\{1\}$
and  $f(G\setminus UT_1)=\{0\}$. For each $m\leq n$, let $u_{mn}\in
U$ be such that $u_{mn}t_{mn}=s_m t_n$. Let $\mathcal{U}$ be
a free ultrafilter on $\mathbb{N}$ and consider the limits along
$\mathcal{U}$ of the sequences $\{u_{mn} \colon n\geq m\}$ and $\{t_{mn}\colon
n\geq m\}$,
\[\lim_{n,\mathcal{U}}u_{mn}=u_m\in  \overline{U}\quad\text{
 and}\quad \lim_{n,\mathcal{U}}t_{mn}=x_m\in \overline{T}^{\luc}.\]
 Choose another free ultrafilter $\mathcal{V}$ on $\N$
 and the corresponding limits  \[
 \lim_{m,\mathcal{V}} u_m=u\quad\text{ and}\quad\lim_{m,\mathcal{V}}x_m=x.\]
  Then using the joint continuity property in
$G^\luc$, we see that
\begin{align*}
\lim_{m,\mathcal{V}}\lim_{n,\mathcal{U}}
f_{u^{-1}}(s_{m}t_{n})=
&\lim_{m,\mathcal{V}}\lim_{n,\mathcal{U}}f(u^{-1}u_{mn}
t_{mn})=\lim_{m,\mathcal{V}}f^\luc(u^{-1}u_{m}
x_{m})\\&=
f^\luc(u^{-1}ux)=f^\luc(x)=\lim_{m,\mathcal{V}}\lim_{n,\mathcal{U}}
f(t_{mn})=1\end{align*}
 since, once $m$ is fixed,
all but at most finitely many $n$'s satisfy
$n>m$; while
\[\lim_{n,\mathcal{U}}\lim_{m,\mathcal{V}} f_{u^{-1}}(s_{m}t_{n})=
\lim_{n,\mathcal{U}}\lim_{m,\mathcal{V}} f(u^{-1}s_{m}t_{n})=0\]
since,  by (4),
 $u^{-1}s_{m}t_{n}\notin UT_1$ whenever $m
 >n$.
Therefore, $f_{u^{-1}}$ is not weakly almost periodic and neither is
$f.$
\end{proof}

 \begin{corollary}
   \label{cor:wap2}
   Let $G$ be a locally compact group.  Let $T$ be a
right  $U$-uniformly discrete subset of $G$ for some
neighbourhood $U$ of $e$. If $T$ is an approximable
$\wap(G)$-interpolation set, then
there is a relatively compact neighbourhood $V$ of $e$, $V\subseteq U$
such that
 $VT$ is translation-compact.
 \end{corollary}

\begin{proof} Suppose that $VT$ is not right translation-compact for
any relatively compact neighbourhood $V$ of $e$ with $V\subseteq U$.
To check that $T$ is not an approximable
 $\wap(G)$-interpolation set, let $V_1,\,V_2$ be open, relatively compact
 neighbourhoods of $e$ with $\overline{V_1}\subseteq V_2
 \subseteq U$.
Since $V_2T$ is not right translation-compact, we may apply
 Lemma \ref{Ruppert} to find $T_1\subseteq T$
  such that no bounded function  $f\colon G\to \C$ with
  $f(T_1)=\{1\}$
  and $f(G\setminus V_2T_1)=\{0\}$ is  weakly almost periodic. Therefore,
   $T$ is not an approximable $\wap(G)$-interpolation set.

The argument is symmetric if we suppose that $VT$ is not left translation-compact
for any neighbourhood $V$ of $e.$
\end{proof}

    \begin{corollary}\label{wap_met} Let $G$ be a  metrizable locally compact group.
If $T$ is an  approximable $\wap(G)$-interpolation set,
then $T$ is right $U$-uniformly discrete and  $UT$ is   translation-compact
for some neighbourhood $U$ of the identity.
    \end{corollary}

    \begin{proof} This is an immediate consequence of Theorem \ref{lem:lucintconv} and Corollary \ref{cor:wap2}.
    \end{proof}

The following result may be known at least in the discrete case, but
    we have found no reference to it in the literature.

    \begin{corollary} Let $G$ be a  metrizable locally compact Abelian group.
If $T$ is a $B(G)$-interpolation set (i.e., a topological Sidon
set), then $T$ is right $U$-uniformly discrete and  $UT$ is   translation-compact
for some neighbourhood $U$ of the identity.
   In particular, Sidon subsets of discrete groups are
translation-finite.
    \end{corollary}

    \begin{proof} $B(G)$-interpolation sets are,
     by Proposition \ref{prop:bg}, approximable $\B(G)$-interpo\-lation
     sets,
      and so   they are approximable $\wap(G)$-interpolation sets.
Corollary \ref{wap_met} finishes the proof.
    \end{proof}

\begin{theorem}\label{inherit} Let $G$ be a locally compact $E$-group
with identity $e$, let $T$ be an $E$-set  which is right uniformly
discrete with respect to some neighbourhood $U$ of $e$ and let  $V$
be   a neighbourhood of $e$ with $V^2\subseteq U$. Then the following
statements are equivalent.
\begin{enumerate}
\item Every function in $\uc(G)$ which is supported in $VT$ is in $\wap(G)$.
\item For every neighbourhood $W$ of $e$ with $\overline W\subseteq V$, every function in $\uc(G)$ which is supported in $W T$ is in $\wap(G)$.
\item $VT$ is right translation-compact.
\item For every neighbourhood $W$ of $e$ with $\overline W\subseteq V,$ $WT$ is right translation-compact.
\end{enumerate}
\end{theorem}

\begin{proof}
The implication (i) $\Longrightarrow$ (ii) is clear.

For the converse, let  $f\in \uc(G)$ be supported in $VT$. Since
$f\in\uc(G),$ we can find for each $n\in\N$, open sets $W_n$ and
$V_n$ such that $\overline{W_n}\subseteq V_n\subseteq \overline{
V_n}\subseteq V$, and $|f (x)| \le\frac 1n$ whenever $x\in V
T\setminus W_n T$. To obtain such neighbourhoods, we take for each
$n\in \mathbb{N}$, a neighbourhood $U_n$ of the identity such that
$|f (x)-f (y)| <\frac 1n$ whenever $xy^{-1}\in U_n$. Then let $W_n =
V\setminus \overline{U_n (V^2\setminus V )}.$ It is clear that $u\in
U_n$ implies $|f (uyt)- f (yt)| < \frac1n$ for every $t\in  T$  and
$y\in G.$  But if $y\in V^2\setminus V$ then $f (yt) = 0$ because
$T$ is $V^2$-uniformly discrete and $f$ is supported in $VT$. It
follows that $|f (ut)| < \frac1n$ whenever
 $u\in U_n(V^2\setminus V)$ and $t\in T.$ Now note that if $v\in V$ and $t\in T$  are such
that $vt\in VT\setminus W_n T,$ then $v\in \overline{U_n
(V^2\setminus V)}$, and therefore $|f (vt)| \le\frac1n.$ Note in
addition that $\overline{W_n}\subseteq U.$ For if $x\in
\overline{W_n}$ then $x\notin U_n (V^2\setminus V)$, and in
particular $x\notin V^2\setminus V.$ Since $x\in\overline V
\subseteq V^2$, we conclude that $x\in V.$ Now it is easy to find a
neighbourhood $V_n$ of the identity such that $\overline{W_n}\subseteq
V_n\subseteq\overline {V_n}\subseteq  V.$

Having obtained these two families $\{W_n\}_{n\in\mathbb N}$ and
$\{V_n\}_{n\in\mathbb N}$ of neighbourhoods of $e$, we let for each $n\in\mathbb N$,
$\varphi_n\in \uc(G)$ be such that $0 \le\varphi_n(x)\le 1$ for all
$x\in G,$ $\varphi_n= 1$ on $W_n$ and $\varphi_n = 0$ on $G\setminus
V_n$. Then, since $T$ is an $E$-set, we see that
 for every $n\in \mathbb N,$ the function
 $1_{T,\varphi_n}$ is in $\uc(G)$ and so
 is $\Phi_n = 1_{T,\varphi_n} f.$ Since the latter function is supported on $V_n T$
and $\overline{V_n}\subseteq  V$, we have by hypothesis that
$\Phi_n\in\wap(G)$. It only remains to observe that $\|f-
\Phi_n\|_\infty\le\frac 1n.$ To see this, note that $\Phi_n$ and $f$
only differ on $(V \setminus W_n)T$. But, if $vt\in V T \setminus
W_n T$, then $|f (vt)| \le\frac1n$ and so \[ |f (vt) - \Phi_n (vt )|
= |f (vt)| {\cdot} |1 - \varphi_n (v)| \le \frac1n.\]
Thus, $f\in\wap(G)$ as required for Statement (i).
\medskip

To prove that (i) implies (iii), assume (i) and suppose for a
contradiction  that $VT$ is not translation-compact. Let then $T_1$
be the set provided by Lemma \ref{Ruppert} for $V$. Then we take
any right uniformly continuous function $\varphi$ with support
contained in $V$ and value $1$ on $e$ and consider the function
$f=1_{T_1,\varphi}$.
Then $f$ is in $\uc(G)$ by (iii) of Lemma \ref{lem:tcc},  has its
support contained in $ VT_1$ and takes the value 1 on $T_1$.   Lemma
\ref{Ruppert} proves that $f$  cannot be weakly almost periodic.
This contradiction with (i) shows that $VT$ must be
translation-compact.

\medskip
In the same way, we prove that (ii) implies (iv).

Lemma \ref{lem:tc} proves  that (iii) implies (i) and (iv) implies
(ii).
\end{proof}

\begin{theorem}\label{wap} Let $G$ be a locally compact $E$-group
with identity $e$, let $T$ be an $E$-set  which is right uniformly
discrete with respect to some neighbourhood $U$ of $e$ and let  $V$
be   an open,  relatively compact   neighbourhood of $e$ with $V^2\subseteq U$. Then the
following statements are equivalent.
\begin{enumerate}
\item Every function in $\uc(G)$ which is supported in $V T$ is in $\wap(G)$.
\item $VT$ is right translation-compact.
\item  For every $\varphi\in\uc(G)$ with support contained in $V$, the function $1_{T,\varphi}$ is in  $\wap(G)$ and its extension $1_{T,\varphi}^\wap$ to $G^\wap$ is zero on $G^*G^*$.
\item For every neighbourhood $W$ of
 $e$ with $W\subseteq V$,
  the set $W\overline{ T}^{\wap}$ is open in $G^\wap$,  and $\overline{T}^{\wap}\subseteq G^{\wap}\setminus G^*G^*$ (here $\overline T^\wap$ is the closure of $T$ in $G^\wap$).
\item $T$ is a $\wap(G)$-interpolation set, for
 every $\varphi\in\luc(G)$ with support contained in $\overline V,$ the function $1_{T,\varphi}$ is in $\wap(G)$
and $\overline V\overline{T}^{\wap}=\overline V\times \beta T$, where $\beta T$ is the Stone-\v Cech compactification of
$T.$
\end{enumerate}
Moreover,
 $T$ is an approximable $\wap(G)$-interpolation set if and
if and only if every neighbourhood  of $e$ contains a second relatively compact neighbourhood
of $e$ for which one (and so all) of the above statements holds.
\end{theorem}

\begin{proof}
Theorem \ref{inherit} proves that \emph{Statements (i) and (ii) are
equivalent}.
\medskip

(ii) $\Longrightarrow$ (iii). Since  the function $1_{T,\varphi}$ is in $\uc(G)$ and is
supported in $VT$, (iii) follows from (i) and (ii) along with  Lemma
\ref{lem:tc}.
\medskip

(iii) $\Longrightarrow$ (iv) We first check that $V\overline{T}^\wap$ is
open in $G^\wap$  (this is inspired by a proof done by Pym on
$G^{\luc}$ in \cite{Pym}). Let $vx$ be any point in $V\overline{
T}^\wap$, let $V_0$ and $V_1$ be neighbourhoods of $e$ such that
$V_0v\subseteq V$ and $\overline{V_1}\subseteq V_0$. Let
$\varphi:G\to[0,1]$ be a continuous function with values $1$ on
$\overline{V_1}$ and $0$ outside of $V_0$, and let  $f$ denote the
right translate by $v^{-1}$ of the function $1_{T,\varphi}$; that
is,
\[f(s)=\sum_{t\in T}\varphi(st^{-1}v^{-1}).\]
Since,  by assumption,   $1_{T,\varphi}\in\wap(G)$, $f\in\wap(G)$ as
well. Let then $f^\wap$ be the continuous extension of $f$ to
$G^\wap$, and put \[V_2=\{u\in V:\varphi(uv^{-1})>1/2\}.\]
 Then $V_2\overline{
T}^\wap$ is open in $G^\wap$ since $V_2\overline{ T}^\wap={\left(
f^{\wap}\right)}^{-1}(]1/2,1])$. Therefore, $V_2\overline{ T}^\wap$
is a neighbourhood of $vx$ in $G^\wap$ which is contained in
$V\overline T^\wap.$ This shows that $V\overline{ T}^\wap$ is open in
$G^\wap.$ The same argument applies to show that $W\overline{
T}^\wap$ is open in $G^\wap$ whenever $W$ is a neighbourhood of $e$
contained in $V$.

For the second part of Statement (iv), let $V_1$ and $\varphi$ be
again as before. Clearly, if $ut\in V_1T$, then
$1_{T,\varphi}(ut)=\varphi(u)=1$ and so $\overline{V_1T}^\wap
\subseteq G^\wap \setminus G^*G^*$
 since, by assumption, the function
$1_{T,\varphi}^\wap$ is zero on $G^*G^*$. It follows in particular
that $\overline{T}^\wap \subseteq G^\wap\setminus G^*G^*$.
\medskip

(iv) $\Longrightarrow$ (i). In fact, only the part $\overline{
T}^\wap\cap G^*G^*=\emptyset$ of Statement (iv)
 is needed to deduce
Statement (i). Let $f\in \uc(G)$ be supported in $VT$ and let
$(s_n)$ and $(t_m)$ be sequences in $G$ such that the limits
\[\lim_n\lim_mf(s_nt_m)\quad\text{ and}\quad \lim_m\lim_nf(s_nt_m)\] exist.
Suppose that $\lim_n\lim_mf(s_nt_m)\ne0.$ Then there is $n_0$ such
that for every $n\ge n_0$ there is $m(n)$ such that $f(s_nt_m)\ne 0$
for every $m\ge m(n),$ and so $s_nt_m\in VT$ for every $n\ge n_0$ and
$m\ge m(n).$ It follows that if $x$ and $y$ are cluster points,  in
$G^\wap$, of $(s_n)$ and $(t_m)$, respectively, then $xy\in
\overline{VT}^\wap=\overline{V}\,\overline{ T}^\wap.$ Accordingly,
either $x$ or $y$ must be in $G$ since by assumption $\overline{
T}^\wap\cap G^*G^*=\emptyset$. However, if $x\in G$ (say), then
\[\lim_m\lim_nf(s_nt_m)=\lim_mf(xt_m)=\lim_n\lim_mf(s_nt_m).\] If
$\lim_m\lim_nf(s_nt_m)\ne0,$ we argue in the same way. The case
\[\lim_n\lim_mf(s_nt_m)= \lim_m\lim_nf(s_nt_m)=0\] is trivial.
Therefore, $f\in\wap(G)$, as required for Statement (i).
\medskip

\emph{We have thus far proved that Statements (i) through (iv) are
equivalent.}
\medskip

(i) $\Longleftrightarrow$ (v).  Suppose that every function in
$\uc(G)$ and  supported in $VT$ is weakly almost periodic.

Let $f:T\to\mathbb C$ be any bounded function on $T$ and extend $f$
to $G$ by the function $f_{T,\varphi}$, where $\varphi\in \luc(G)$,
$\varphi(e)=1$ and the support of $\varphi$ is contained in $V$.
Since $f_{T,\varphi}\in \uc(G)$ as seen in
 Lemma \ref{lem:tcc} (iii) and is supported in $VT$,  $f_{T,\varphi}\in \wap(G)$ by
assumption. This shows that any bounded $f$ on $T$ extends to a
weakly almost periodic function on $G$, and so $T$ is a
$\wap(G)$-interpolation set. This also shows in particular that
$1_{T,\varphi}\in\wap(G)$ as required for the first part of the
statement.

  Since $T$ is a $\wap(G)$-interpolation set, $\overline{ T}^\wap$
  is homeomorphic to $\beta T$, with the points of $T$  fixed by the homeomorphism.
  So we may identify $\overline{T}^\wap$ and $\beta T$.
Then, using the joint continuity on $G\times G^\wap,$ we may consider
the continuous surjection $\overline{V}\times \beta T\to
\overline{VT}^\wap$ which extends the multiplication mapping
$V\times T\to VT$. We prove that this extension is injective. Let
$f_{T,\varphi}^\wap$ and $f^{\beta T}$ be the extensions of
$f_{T,\varphi}$ and $f$ to $G^\wap$ and $\beta T,$ respectively, and
let $(u,x)\in U\times \overline{T}^\wap$. Then, since $T$ is right
uniformly discrete with respect to $U$, we have
\begin{equation}\label{main}
f_{T,\varphi}^\wap(ux) =\lim_\alpha\sum_{t\in
T}f(t)\varphi(ut_\alpha t^{-1})=\lim_\alpha f(t_\alpha)\varphi(u)=
 f^{\beta }(x)\varphi(u) \qquad(\ast\ast)\end{equation}
(here $(t_\alpha)$ is a net in $T$ converging to $x$ in $\beta T$ when $x$ is regarded as a point in $\beta$, and converging to $x$ in $G^\wap$ when $x$ is regarded as a point in $\overline{T}^\wap$).
Accordingly, if $(u,x)\ne (v,y)$ in $\overline{V}\times \beta T,$
then the functions $\varphi$ and $f$ may be chosen in \eqref{main}
so that $f_{T,\varphi}^\wap$ separates $ux$ and $vy$ in $G^\wap.$ In
fact if $x= y$ then $u$ and $v$ must be distinct and so we may choose   $\varphi\in\uc(G)$ with support
contained in $V$, $\varphi(u)=1$ and $\varphi(v)=0$. If
$x\ne y$ we may choose $f\in\ell_\infty(T)$ such that $f^\beta(x)=1$
and $f^\beta(y)=0$ and $\varphi\in\uc(G)$ with value $1$ on $\overline V.$
We obtain in
all the cases,
\[f_{T,\varphi}^\wap(ux)=f^\beta(x)\varphi(u)=1\quad\text{ while}\quad f_{T,\varphi}^\wap(vy)=f^\beta(y)\varphi(v)=0.\]
Hence, $ux\ne vy$ in $G^\wap.$
Therefore, $\overline V\times \beta T\to \overline{ VT}^\wap$ is
injective, and so it is a  homeomorphism.

For the converse, we apply Theorem \ref{inherit}. Let $W$ be any
neighbourhood of $e$ with $\overline W\subseteq  V,$ and let $f$ be
any function in $\uc(G)$ supported in $W T.$ Let $f^\luc$ be the
continuous extension of $f$ to $G^\luc,$ and consider its
restriction $f\restr_{\overline{VT}^\luc}$.
  Since $\overline{V}\times \beta T$ and $\overline{VT}^\luc$ are homeomorphic by
\cite{Pym}, and $\overline V\times \beta T$ and $\overline{VT}^\wap$
are homeomorphic by assumption, we may regard
$f\restr_{\overline{VT}^\luc}$
 as a continuous function on
$\overline{VT}^\wap$, and extend it to a continuous function
$\widetilde f$ on $G^\wap.$ Note now that
$\overline{V}^\wap=\overline{V}^\luc=\overline V,$ so
\[\widetilde
f(vt)=f\restr_{\overline{VT}^\luc}(vt)=f^\luc(vt)=f(vt)\quad\text{for
every}\quad v\in \overline V, t\in T.\]

Taking now $\varphi\in\luc(G)$  with support contained in $V$ and
$\varphi(W ) = 1$, we see that $\widetilde f 1_{T,\varphi} = f$.
Indeed, both functions obviously coincide on every  $s\notin V T$.
If $s\in W T$, we have $1_{T,\varphi}(s) = 1$ and so $\widetilde f(s)1_{T,\varphi}(s)=\widetilde f(s)=f(s).$
Finally, if
$s \in V T$ but $s \notin W T$ , then  $\widetilde f(s)1_{T,\varphi}(s)=f(s)1_{T,\varphi}(s)=0=f(s).$

Since $1_{T,\varphi}\in \wap(G)$ by (v),  we have $f \in \wap(G).$
We conclude, with  Theorem \ref{inherit}, that every function in
$\uc(G)$ which is supported in $V T$ is in $\wap(G)$. Hence,
Statement (i) holds.
\medskip

\emph{Therefore, Statements (i)--(v) are equivalent}.
\medskip

We  now prove the last statement. Necessity has been already proved in Corollary \ref{cor:wap2}.

The converse  is an easy check. Fix, to begin with, a bounded
function $f\colon T\to \C$ and a  neighbourhood $U_0$ of the identity.
There is then a relatively compact neighbourhood $V_0\subseteq U_0\cap V$ such that any
$\uc(G)$-function supported in $V_0T$ is weakly almost periodic.
Choose another neighbourhood $V_1$ of the identity with
$\overline{V_1}\subseteq V_0$ and let $T_1\subseteq T$. Consider as well
a function $\psi\in \luc(G)$ with $\psi(V_1)=\{1\}$ and
$\psi(G\setminus V_0)=\{0\}$. By Lemma \ref{lem:tcc} the functions
$f_{T,\psi}$ and $1_{T,\psi}$ are both in $\uc(G)$. Since both of
them are supported in $V_0T$, we have, by assumption, that both
$f_{T,\psi}$ and $1_{T_1,\psi}$ are weakly almost periodic. Since
$f_{T,\psi}$ coincides with $f$ on $T$ and $f$ was arbitrary, $T$ is
a $\wap(G)$-interpolation set. Since $1_{T_1,\psi}(V_1T_1)=\{1\}$
and $1_{T_1,\psi}(G\setminus V_0T_1)=\{0\}$, we see that $T$ is
indeed approximable.
\end{proof}

The whole  Theorem 7 of \cite{rup} and also Proposition 2.4 of
\cite{C4} now follow from Theorem \ref{wap} when $G$ is discrete. We
emphasize here that fact for future reference.
 \begin{corollary}[Theorem 7 of \cite{rup} and Proposition 2.4 of
 \cite{C4}]\label{cor:wap}
 Let $G$ be discrete. A subset $T\subseteq G$ is an approximable
 $\wap(G)$-interpolation set if and only if it is translation-finite
 \end{corollary}

\begin{corollary}
Let $G$ be a metrizable $E$-group and let $T\subseteq G$ be an
$E$-set. Then $T$ is an approximable $\wap(G)$-interpolation set if
and only if $T$ is right uniformly discrete with respect to some
neighbourhood $V$ of $e$ such that  $VT$ is translation-compact.
 \end{corollary}
\begin{proof}
If $T$ is an approximable $\wap(G)$-interpolation set, then it is an
approximable $\luc(G)$-interpolation set. So by Theorem
\ref{lem:lucintconv}, $T$ is right uniformly discrete with respect to some
neighbourhood $V$ of $e$. By Theorem
\ref{wap}, we may choose $V$ such that $VT$ is translation-compact.

The converse follows from Theorem \ref{wap}.
\end{proof}

In general when $G$ is not an $E$-group,  the theorem fails even if $G$ is metrizable.

\begin{remark}
Let  $G=SL(2,\R)$ and $VT$ be a $t$-set in $G$  as constructed in
Example \ref{$t$-set}. Then $T$ is right uniformly discrete,  $VT$ is translation-compact
but $T$ is not an approximable $\wap(G)$-interpolation set
since $SL(2,\R)$ is minimally weakly almost periodic group, that is,
$\wap(G)=\C\oplus C_0(G)$.
\end{remark}

Before we prove our next main theorem in this section, we need the following lemmas, the proof of  the second one relies on a Ramsey theoretic theorem of Hindman.
Recall from \cite[Definition 5.13]{HS}
that if $(x_n)$ is a a sequence in $G$ then the
sequence $(y_n)$ is a {\it product subsystem of $(x_n)$} if and only if there is a sequence
$(H_n)$ of finite subsets of $\N$  such that for  every $n\in \N$,
\[\max H_n < \min H_{n+1}\quad\text{ and}\quad y_n = \Pi_{t\in H_n}x_t, \]
where $\Pi_{t\in H_n}x_t$ is used to denote the product in decreasing order of indices, contrarily to what is chosen in \cite{HS}. So, for instance, if $H_n=\{2,6,23\}$, then $\Pi_{t\in H_n}x_t=x_{23}x_6 x_2$.

We will also need to use  the finite product set $FP((x_n)_n)$ associated to a sequence $(x_n)_n\subseteq G$. It is is defined as
\[ FP((x_n)_n)=\left\{\prod_{n\in F} x_n \colon F\in \N^{<\omega}\right\}.\]

\begin{lemma}\label{subsystem} If $G$ is a non-compact locally compact group and $U$ is  a relatively compact symmetric neighbourhood
of $e$, then $G$ contains a sequence $(x_n)$ such that every product subsystem of $(x_n)$ is $U$-uniformly discrete.
\end{lemma}

\begin{proof} Start by fixing an arbitrary $x_1$ in $G.$ Once $x_1,...,x_n$ have been chosen, consider the finite set
\[F_p=\{\Pi_{t\in H}x_t: H\subseteq \{1,...,p\}\}.\]
Then choose \[x_{n+1}\notin F_nU^2F_n^{-1}\cup F_n^{-1}U^2F_n.\]

Suppose now that $(y_n)$ is a product subsystem of $(x_n).$
If $(y_n)$ is not $U$-uniformly discrete, we can find $n<m$ such that $y_mU\cap y_nU\ne\emptyset,$ that is,
$y_m\in y_nU^2$. Now there exist $k(n)$ and  $k(m)$ such that $k(n)<k(m)$, $y_n\in F_{k(n)}$ and $y_m\in F_{k(m)}$.
We may then write  $y_m=x_{k(m)}z_m$ for some $z_m\in F_{k(m)-1}.$
Since $y_m\in y_nU^2,$ it follows that $x_{k(m)}\in F_{k(n)}U^2F_{k(m)-1}^{-1},$ which goes against our
construction.
\end{proof}

\begin{lemma}\label{Hindman} A non-compact locally compact group cannot be the union of finitely many right (respectively, left) translates of a left (respectively,  right) translation-compact set.
\end{lemma}

\begin{proof} Suppose that $G=\bigcup_{j=1}^pTt_j$ for some $T\subseteq G$ and $t_1, t_2,..., t_p\in G.$
Let $(x_n)_n$  be the sequence constructed in Lemma \ref{subsystem}. Then \[FP((x_n)_n)\subseteq \bigcup_{j=1}^p(FP((x_n)_n)\cap Tt_j).\] By \cite[Corollary 5.15]{HS}, there exists a product subsystem $(y_n)_n$ of $(x_n)_n$ with $FP((y_n)_n)$ contained in $Tt_j$ for some $j=1,2,...,p.$ Let now $L=\{y_nt_j^{-1}:n\in\mathbb N\}$, let $\{y_{n_1}{t_j}^{-1}, y_{n_2}t_j^{-1},..., y_{n_k}t_j^{-1}\}$ be any finite subset of $L$ and let $m$ be any integer greater than $\max\{n_1,n_2,...,n_k\}$.
 Then, since  $y_m=(y_my_{n_i})y_{n_i}^{-1}$ and $y_my_{n_i}\in FP((y_n)_n)$   for each $i=1,2,...,k,$ we see that
 \[y_m\in Tt_jy_{n_1}^{-1}\cap Tt_jy_{n_2}^{-1}\cap...\cap Tt_jy_{n_k}^{-1}.\]
 By Lemma \ref{subsystem}, the set $\{y_m:m>\max\{n_1,n_2,...,n_k\}\}$ is uniformly discrete,
 hence cannot be relatively compact, and so
 $T$ is not left translation-compact.
\end{proof}

\begin{theorem}\label{wap=wap_0}
Let $G$ be a locally compact $E$-group with a neighbourhood $U$ of
the identity and a right  $U$-uniformly discrete $E$-set $T$.  Then
Theorem \ref{wap} holds with $\wap_0(G)$ replacing $\wap(G).$ In
particular, the following statements are equivalent.
 \begin{enumerate}
  \item $T$ is  an approximable $\wap(G)$-interpolation set.
 \item $T$ is  an approximable $\wap_0(G)$-interpolation set.
  \item There exists a neighbourhood $V$ of $e$, $V^2\subseteq U$ for which $VT$ is translation-compact.
  \end{enumerate}
\end{theorem}

\begin{proof} We only need to prove the implication
(iii) $\Longrightarrow$ (ii). Let $V$ be a neighbourhood of $e$ such
that $V^2\subseteq U$ and $VT$ is translation-compact, and let $f$
be any function in $\uc(G)$ which is supported in  $VT.$
 We show that $f\in \wap_0(G)$, this will imply that $T$ is
 an approximable  $\wap_0(G)$-interpolation set exactly as   in Theorem \ref{wap}.
By Theorem \ref{wap}, $f\in \wap(G)$. Let $\mu\in\wap(G)^*$ be the
unique invariant mean on $\wap(G)$. By Ryll-Nardzewski's theorem,
$\mu(|f|)$ is the unique constant in the closed convex hull
$co({}_G|f|)$ of ${}_G|f|$, see \cite[Theorem 1.25 and Corollary
1.26]{Bu}. Here, ${}_sf$ is the right translate of $f$ by $s.$
 Suppose that $\mu(|f|)>0$, and let $\sum_{k=1}^nc_k({}_{x_k}|f|)\in co({}_G|f|)$ be
such that \[\|\sum_{k=1}^nc_k({}_{x_k}|f|)-\mu(|f|)\|=\sup_{s\in
G}|\sum_{k=1}^nc_k|f|(sx_k)-\mu(|f|)|<\frac{\mu(|f|)}2.\]
 It follows that $G=\bigcup_{k=1}^nVTx_k^{-1}$; otherwise,
 if some element $s$ of $G$ is not in $\cup_{k=1}^nVTx_k^{-1},$ then
$f(sx_k)=0$ for every $k=1,2,...,n$ since $f$ is supported in $VT$,
and so
 \[\frac{\mu(|f|)}2>
 \|\sum_{k=1}^nc_k({}_{x_k}|f|)-\mu(|f|)\|\ge |\sum_{k=1}^nc_k|f|(sx_k)-\mu(|f|)|=
 \mu(|f|),\]
 which is absurd.
Since $G$ is not compact and $VT$ is translation-compact, this is not possible by
 Lemma \ref{Hindman}. This  implies  that $m(|f|)$
must be zero, and so $f\in\wap_0(G),$ as required.
\end{proof}

\begin{corollary} \label{noapap}Let $G$ be a locally compact group.
Then no right uniformly discrete subset of $G$ can be an
approximable $\ap(G)$-interpolation set.
\end{corollary}

\begin{proof} Let $T$ be a subset of $G$
which is right uniformly discrete  with respect to some
neighbourhood $U$ of $e$ and suppose that $T$  is an approximable
$\ap(G)$-interpolation set.
 By Corollary \ref{cor:wap2}, there is $V\subseteq U$ such
that $VT$ is translation-compact, and so by Theorem \ref{wap=wap_0}, every $\uc(G)$-function supported on $VT$ is  in $\wap_0(G)$. (Observe that the $E$-property is not needed to prove neither Corollary \ref{cor:wap2} nor the implication (iii) $\Longrightarrow$ (ii) in Theorem \ref{wap=wap_0}.)
 Pick  a nonzero $f\in \ap(G)$
supported on $VT$.
 Then  $|f|\in \ap(G)\cap \wap_0(G)$, but this implies $f=0$
 for $\mu$ coincides with the Haar measure on $G^\ap$
 and
$|f|$ would extend to a  non-zero, continuous and   positive
function in $G^\ap$.
\end{proof}

\begin{corollary}\label{noapmet}
   Let $G$ be a metrizable locally compact group. Then no subset of $G$ can be an approximable $\ap(G)$-interpolation set.
\end{corollary}
    \begin{proof}
  We assume that $G$ is not compact, as compact metrizable groups
  cannot contain infinite $\ap(G)$-interpolation sets.

      Suppose $T$ is an approximable $\ap(G)$-interpolation set.
By Theorem \ref{lem:lucintconv}, $T$ is right $U$-uniformly discrete
for some neighbourhood $U$ of $e$. This contradicts Corollary
\ref{noapap}.

\noindent
{\it Second proof.}
We can also argue directly without using Theorem \ref{wap=wap_0}.
By Corollary \ref{cor:wap2}, we may suppose that $UT$ is also right
translation-compact. Let $V_1$ and $V_2$ be two neighbourhoods of
the identity with $\overline{V_1}\subseteq V_2\subseteq U$  and $h\in
\ap(G)$ be such that $h(V_1T)=\{1\}$ and $h(G\setminus V_2T)=\{0\}$.
Let $(x_n)_n$ be a sequence in $G$ that goes to infinity. Since the
set of left translates $\{h_{x_n}\colon n<\omega\}$ is relatively
compact in $\ell_\infty(G)$, we can assume by taking the tail of a
subsequence, if necessary, that
\[\|h_{x_n}-h_{x_m}\|_\infty<1, \mbox{ for all $n,m$.}\]
As a consequence $|h(x_ns)-h(x_ms)|<1$ for all $n,m$ and $g\in G$.
It follows that
\begin{equation}\label{eq:tc}x_n^{-1}V_1T\subseteq x_m^{-1}V_2T, \mbox{ for all $n,m$.}\end{equation}

Since $UT$ is right translation-compact, there must be a finite
family $x_{n_1},\ldots,x_{n_k}$,  such that
$x_{n_1}^{-1}UT\cap\ldots \cap x_{n_k}^{-1}UT$ is  relatively
compact. But an application of \eqref{eq:tc} shows that
\[x_{n_1}^{-1}V_1T\subseteq x_{n_1}^{-1}UT\cap\ldots \cap x_{n_k}^{-1}UT.\]
Since $x_{n_1}^{-1}V_1T$ is not relatively compact,
$x_{n_1}^{-1}UT\cap\ldots \cap x_{n_k}^{-1}UT$ cannot be relatively
compact either. This contradiction proves the corollary.
    \end{proof}

\section{On the union of approximable $\A(G)$-interpolation sets}

As already noted in the introduction, in a discrete Abelian group
any finite union of Sidon sets is a Sidon set \cite{drury70}. A finite union of $I_0$-sets is however not always
an $I_0$-set, for example  the union of the two $I_0$-sets  $\{6^n:n\in\mathbb N\}$ and $\{6^n+n:n\in\mathbb N\}$ is not an $I_0$-set (see for instance \cite[p. 132]{kaha70}).

The property is not true for right $t$-sets either, simply take $T$ as any right
$t$-set, $s\ne e$ in $G$ and consider $T\cup sT$. As we show next
finite unions of right uniformly discrete sets are not in general uniformly discrete either. As a matter of fact, this will be   the only
obstacle towards union theorems for approximable interpolation sets,
see Proposition \ref{unionunif} below.

Any finite union of right translation-finite sets stays right
translation-finite, this was obtained by Ruppert as a consequence of
translation-finite subsets being approximable
$\wap(G)$-interpolation sets for discrete $G$ (see Theorem \ref{wap}
and Corollary \ref{cor:wap}). Independently, this result was also
proved directly from the definition using combinatorial arguments in
\cite[Lemma 5.1]{FPr}.

This is generalized in Corollary \ref{t-cunion} below, where $G$ is a locally compact $E$-group.
We shall in fact deduce from
Theorem \ref{wap} that finite unions of translation-compact sets of
the form $VT$, where $V$ a relatively compact neighbourhood of the
identity and $T$ is a $V^2$-right uniformly discrete, stay
translation-compact under the  condition that the union of the sets $T$ is right (left) uniformly discrete.

\begin{proposition} \label{unionunif} Let $G$ be a locally compact group. Finite unions of right (left) uniformly discrete sets are right (left) uniformly discrete if and only if $G$ is discrete.
\end{proposition}

\begin{proof}  Suppose that $G$ is not discrete, and
fix relatively compact neighbourhoods $U$ and $V$ of $e$ with
$V^2\subseteq U$. Then by \cite[ Lemma 5.2]{FG1},  $V$ contains a
faithfully indexed sequence $S=\{s_n \colon n<\omega\}$ converging
to the identity. Let $T_1=\{t_n\colon n<\omega\}$ be a $U$-right
uniformly discrete set. Then $T_1$ and $T_2=\{s_nt_n\colon
n<\omega\}$ are both $V$-right uniformly discrete sets but $T_1\cup
T_2$ is not.
\end{proof}
%
%
%

\begin{proposition}\label{prop:unionapp}
  Let $G$ be a topological group and let
  $\A(G)$ be a subalgebra of $\luc(G)$. If  $T_1$ and $T_2$ are approximable
  $\A(G)$-interpolation sets and $T_1\cup T_2$ is right
   uniformly discrete,
  then $T_1\cup T_2$ is an approximable $\A(G)$-interpolation set.
\end{proposition}
\begin{proof}
Fix a neighbourhood $W$ of the identity such that $T_1\cup T_2$
  is right $W$-uniformly discrete.
Let  $U$ be an  arbitrary neighbourhood of the identity.

   Since $T_1$ and $T_2$ are approximable
  $\A(G)$-interpolation sets there are two pair of neighbourhoods of
  the identity $V_{11}$, $V_{12}$, and $V_{21}$, $V_{22}$ with
  $\overline{V_{11}}\subseteq V_{12}\subseteq W\cap U$ and   $\overline{V_{21}}
  \subseteq  V_{22}\subseteq W\cap U$ with the properties stated in Definition \ref{def:main} of
approximable interpolation set .
For each  $S\subseteq T_1\cup T_2$, we obtain from the definition
 two functions $h_1, h_2\in \A(G)$ such that
\begin{align*} h_1(V_{11}(S\cap T_1))=&\{1\},\qquad
h_2(V_{21}(S\cap(T_2\setminus T_1)))=\{1\},\\ h_1\bigl(G\setminus
V_{12}(S\cap T_1)\,\bigr)=&\{0\}\quad\text{ and}\quad
h_2\bigl(G\setminus V_{22}(S\cap(T_2\setminus
T_1))\,\bigl)=\{0\}.\end{align*} If $S\subseteq T_1$ or $S\cap
T_1=\emptyset$ (in which case $S\subseteq T_2$), we just consider
one of the functions, $h_2$ or $h_1$ respectively,.

    We first  prove that $T_1\cup T_2$ is an $\A(G)$-interpolation set.
Let $f\colon T_1\cup T_2\to \C$ be a bounded
    function and assume that the above functions $h_1$ and $h_2$
    have been constructed for $S=T_1\cup T_2$, so that $S\cap
    T_1=T_1$ and $S\cap(T_2\setminus T_1)=T_2\setminus T_1$.
We may consider two functions
  $f_1,f_2\in \A(G)$ such that $f_{1}\restr_{T_1}=f\restr_{T_1}$ and
  $f_2\restr_{T_2}=f\restr_{T_2}$. Since  $T_1\cup T_2$ is
right  $U$-uniformly discrete, \[T_1\subseteq  G\setminus
V_{22}(T_2\setminus T_1)\quad\text{ and}\quad T_2\setminus
T_1\subseteq G\setminus V_{12}T_1,\]
 whence it follows that $f$ coincides with
$f_1h_1+f_2h_2$ on $T_1\cup T_2$. Since $f_1h_1+f_2h_2\in \A(G),$ we
have shown that $T_1\cup T_2$ is an $\A(G)$-interpolation set.

To see that $T_1\cup T_2$ is an approximable $\A(G)$-interpolation set,
let $S$ be again an arbitrary subset of $T_1\cup T_2.$ Then we only
have to observe that  \[(h_1+h_2)((V_{11}\cap
V_{21})S)=\{1\}\quad\text{ and}\quad (h_1+h_2)(G\setminus
(V_{12}\cup V_{22})S)=\{0\}.\] Since $\overline{V_{11}\cap
V_{21}}\subseteq V_{12}\cup V_{22}\subseteq U$,
the proof is done.
  \end{proof}

\begin{corollary}
  Let $G$ be a metrizable topological group, let
  $\A(G)$ be a subalgebra of $\luc(G)$ and let  $T_1$ and $T_2$ be approximable
  $\A(G)$-interpolation sets. Then,
    $T_1\cup T_2$ is an approximable $\A(G)$-interpolation set
  if and only if $T_1\cup T_2$ is right
   uniformly discrete.
\end{corollary}

\begin{proof}
  Sufficiency is proved in Proposition \ref{prop:unionapp} above.
  For the necessity one only has to note that $\A(G)$-interpolation
  sets are $\luc(G)$-interpolations sets, and hence they
  must be right uniformly discrete by
Theorem \ref{lem:lucintconv}.
\end{proof}


The following is a generalization of the result on finite unions of translation-finite sets proved for discrete groups
in \cite{rup} and \cite{FPr}.
Note that by
Example \ref{ex:tc} this cannot be proved with Corollary
\ref{cor:wap} alone.

\begin{corollary}{\label{t-cunion}} Let $G$ be a locally compact $E$-group and
$T_1$ and $T_2$ be subsets of $G$  such that $T_1\cup T_2$ is right uniformly discrete
with respect to some neighbourhood $U$ of $e.$ If $UT_1$ and $UT_2$ are translation-compact, then
$VT_1\cup VT_2$ is translation-compact for some neighbourhood $V$ of $e$.
\end{corollary}

\begin{proof}  By Theorem \ref{wap}, $T_1$ and $T_2$ are approximable $\wap$-interpolation sets.
By Proposition \ref{prop:unionapp}, $T_1\cup  T_2$ is an approximable $\wap$-interpolation set. By Theorem \ref{wap},
$VT_1\cup VT_2$ is translation-compact for some neighbourhood $V$ of $e$.
 \end{proof}

\section{Examples and Remarks}
If $G$ is a topological group and $G_d$ denotes the same group
equipped with the discrete topology, then $\wap(G)=\wap(G_d)\cap
\CB(G)$, $\B(G)=\B(G_d)\cap \CB(G)$ and $\ap(G)=\ap(G_d) \cap
\CB(G)$.
It
could be conjectured that whenever $T$ is both an (approximable)
$\luc(G)$-interpolation set and an approximable $\wap(G_d)$-interpolation, then
$T$ must be an approximable $\wap(G)$-interpolation set. 
 As we show next,  it can be the case that it is not
possible to get the necessary functions   that are \emph{both} continuous and
weakly almost periodic.
\begin{example}\label{ex:tc} There exists a subset $T\subseteq  \R$ that is both uniformly discrete and translation-finite but
such that $UT$ is not translation-compact for any neighbourhood $U$
of the identity. This provides an example of a subset $T\subseteq  \R$
that is an approximable  $\CB(G)$-interpolation set, an approximable
$\wap(G_d)$-interpolation set but yet it is not a
$\wap(G)=\wap(G_d)\cap \CB(G)$-interpolation set.
\end{example}

\begin{proof}
Let $(\alpha_n)_n$ be a decreasing sequence  of real numbers with $\alpha_0=1$
such that the set $\left\{\alpha_n \colon n\geq 0\right\}$ is
linearly independent over the rationals and  $\lim_{n\to \infty}
\alpha_n=0$. We define $T$ as $T=\left\{n+\alpha_n \colon n\in
\N\right\}$.   It is obvious that $T$ is a linearly independent and
uniformly discrete subset of $\R$.
As every linearly independent subset, $T$ is a $t$-set (and so translation-finite). It
is actually  an $\ap(\R_d)$-interpolation set. Indeed, if $T_1$ and $T_2$ are disjoint subsets of $T$ then $\cl_{\R_d^\ap}T_1\cap \cl_{G^\ap}T_2=\emptyset$ follows simply by choosing   a character $\chi\colon G\to \mathbb{T}$ sending $T_1$ to 1 and $T_2$ to $-1$. Then, by for example \cite[Corollary 3.6.2]{enge77}, we see that  $T$ is an $\ap(\R_d)$-interpolation set.

Let now $U$ be a neighbourhood of 0 in $\R$. If
$\alpha_n-\alpha_{n+k} \in U$, then
\[ n+\alpha_n\in -k+T+U.\]
Now, if $F\subseteq \N$ is finite, there will be an infinite subset
$S$ of $\N$ such that $n\in S$ implies that
$\alpha_n-\alpha_{n+k}\in U$ for all $k\in F$. Then
\[ \left\{ n+\alpha_n \colon n\in S \right\} \subseteq
\bigcap\left\{-k+U+T\colon k\in F\right\}.\] Since $ \left\{
n+\alpha_n \colon n\in S \right\}$ is an infinite  uniformly
discrete set, we have that
\[\bigcap\left\{-k+U+T\colon k\in F\right\}\] is not relatively compact.
Therefore,
 $U+T$ is not translation-compact.
%
%
\end{proof}

  Next, we show that an $\luc(G)$-interpolation set
    needs not be uniformly discrete, if $G$  is not metrizable.
    \begin{example}\label{ex:lucint}
    The group $K=\Z^{\ap}$,  the Bohr compactification of the
    discrete group of the integers, contains an
$\luc(G)$-interpolation that is not  uniformly discrete.
 \end{example}
 \begin{proof}
     Let $T\subseteq \Z$ be an infinite
    $\ap(\Z)$-interpolation set (i.e., an $I_0$-set, for instance
    $T=\{2^n\colon n\in \N\}$). Then for
    $\A(K)=\ap(K)=\wap(K)=\luc(K)=\CB(K)$, $T$ is an
    $\A(K)$-interpolation subset of $K$ but, $K$ being compact,  $T$ is not a right (left)
    uniformly discrete subset of  $K$.
 \end{proof}
    Next we use a argument from
     \cite{rams80} and show that under the Continuum Hypothesis (CH) the
      above situation is  universal among  nonmetrizable Abelian groups.
This indicates that,  in Theorem \ref{lem:lucintconv}, metrizability
is a
      hardly avoidable hypothesis.
    \begin{theorem}\label{CH} Let $G$ be locally compact abelian group, and let  $\A(G)=\ap(G)$, $\wap(G)$ or $\luc(G)$.
    Under CH, every $\A(G)$-interpolation set is uniformly discrete if and
    only if $G$ is metrizable.
    \end{theorem}

    \begin{proof} Sufficiency is proved in Theorem \ref{lem:lucintconv} without assuming CH.
    Conversely, suppose that $G$ is not metrizable and let $K$ be a
    compact subgroup of $G$ such that $G/K$ is metrizable
    (this is always available, see Theorem 8.7 of \cite{hewiross1}). Then $K$ is
    not metrizable and therefore its topological weight must be at
    least  $\cc$,
     (it is here where
    we use CH), therefore, by \cite[Theorem 24.15]{hewiross1},  $\widehat{K}$ is an Abelian group of
    cardinality at least $\cc$.

    By taking the subgroup generated by a maximal independent subset of   $\widehat{K}$ containing only elements of the same order \cite[Section 16]{fuchsx}, we see that $\widehat{K}$ contains a subgroup $D$ isomorphic
    to a direct sum $\bigoplus_{\cc} H$ where $H$ is either $\Z$ or a finite group. In both cases $D$ has a quotient
that admits a  compact group topology (if $H=\Z$ we use
\cite[Corollary 4.13]{fuchsx}, and if $H$ is finite then $D$ itself works for
$D\cong H^\omega$). Denote
    the discrete dual of this compact group by $L$. We thus have an injective group homomorphism $j$ and a  surjective group homomorphism $\pi$ as follows  \[  \xymatrix{ \widehat{L}_d & \ar[l]_{\pi}D\ar[r]^j &\widehat{K}}.\]
    Dualizing  the above diagram, we  obtain
    \[  \xymatrix{\widehat{ \widehat{L}_d}     \ar[r]^{\widehat{\pi}}& \widehat{D}&\ar[l]_{\widehat{j}}\widehat{\widehat{K}}},\]
where $\widehat{j}$ is a quotient  homomorphism and $\widehat{\pi}$
is a topological isomorphism of $\widehat{ \widehat{L}_d} $ onto a
subgroup of    $\widehat{D}$. By Pontryagin
    duality (cf.
     \cite[Theorem 26.12]{hewiross1}),
      we can identify $\widehat{\widehat{L}_d}$ with $L^{\ap}$
    and $\widehat{\widehat{K}}$ with $K$. We  obtain in this way
    \[  \xymatrix{L^\ap  \ar[r]^{\widehat{\pi}}& \widehat{D} &\ar[l]_{\widehat{j}}K}.\]

    Now, $L$, as all discrete Abelian groups, admits some infinite
    $\ap(L)$-interpolation set $A$ (see \cite{hartryll64} or
    \cite{galihern99fu}). Arguing as in Example \ref{ex:lucint} we see
    that $A$ is an $\ap(L^\ap)$-interpolation set and, hence, $\widehat{\pi}(A)$ is an  $\ap(\widehat{D})$-interpolation set.
   Now $\widehat{\pi}(A)$ can be lifted through the quotient homomorphism $\widehat{j}$ to obtain an
    $\ap(K$)-interpolation set $I$. As almost periodic functions on
    $K$ extend to almost periodic functions on $G$,  we see that $I$ is
    an $\ap(G)$-interpolation set. Being infinite and contained in a
    compact group it cannot be right (left) uniformly discrete.
    \end{proof}

We would like to finish   pointing some of the questions that have
not been settled in the present paper and in our opinion deserve
further attention.

Sidon sets (i.e. $B(G)$-interpolation sets)
  have been the object of
serious attention in the literature, but  mostly in the discrete and
Abelian cases. The non-discrete, non-Abelian case have  received
little attention. Since $\B(G)$-interpolation sets have not been
under the focus in the present paper, some basic questions are yet
to be clarified. One may ask for instance:

\begin{question}\label{Q1}
Is Proposition \ref{openpb} valid for all  locally compact metrizable
groups?
 In particular,
    are $\B(G)$-interpolation sets necessarily $B(G)$-interpolation sets?
 Another interesting question is whether
all  $\B(G)$-interpolation sets are approximable $\B(G)$-interpolation sets.
\end{question}


Let $G$ be a metrizable locally compact group. While
$\ap(G)$-interpolation sets are never approximable,
$\luc(G)$-interpolation sets are always approximable. Sitting in
between the algebras $\luc(G)$ and $\ap(G)$ we have the algebras
$\B(G)$ and $\wap(G)$. As we have just remarked,
$\B(G)$-interpolation sets are often approximable (at least, when $G$ is
discrete and Abelian they always are) but it is not so clear whether
$\wap(G)$-interpolation sets should be on the approximable side or
not. So one may ask.
\begin{question}
Are $\wap(G)$-interpolation sets approximable? In particular, let
$T\subseteq \Z$ be a  $\wap(\Z)$-interpolation set, must $T$ be
approximable?
\end{question}

\noindent {\sc Acknowledgement.} This paper was written when the
first author was visiting University of Jaume I in Castell\'on in
December 2010-January 2011. He would like to thank Jorge Galindo for
his hospitality and all the folks at the department of mathematics
in Castell\'on. The work was partially supported by Grant
INV-2010-20 of the 2010 Program for Visiting Researchers of
University Jaume I. This support is also gratefully acknowledged.


\begin{thebibliography}{99}



\bibitem{BF} { J. W. Baker and M. Filali,} {\em On the analogue of Veech's theorem in the
WAP-compactification of a locally compact group,} Semigroup Forum 65
(2002), no. 1, 107--112.

\bibitem{BJM} {J. F. Berglund, H. D. Junghenn \and P. Milnes,}
 {\em Analysis on Semigroups: Function Spaces, Compactifications,
Representations,}  Wiley, New York (1989).

\bibitem{Bu} {R. B. Burckel, } {\em Weakly almost periodic functions on semigroups,} Gordon and Breach Science Publishers, New York-London-Paris 1970.


\bibitem{C1} {C. Chou,} {\em Weakly almost periodic functions and Fourier-Stieltjes algebras of
locally compact groups}, Trans. Amer. Math. Soc., 274 (1982), no. 1, 141--157.

\bibitem{C2} {C. Chou,} {\em Weakly almost periodic functions and almost convergent
functions on a group,} Trans. Amer. Math. Soc.
 206 (1975), 175--200.

\bibitem{C4}
C. Chou,
\newblock Weakly almost periodic functions and thin sets in discrete groups,
\newblock {\em Trans. Amer. Math. Soc.}, 321 (1990), no 1., 333--346,

\bibitem{MDG} M. Dechamps-Gondim,
{\em Ensembles de Sidon topologiques}, Annales de l´institut
Fourier, tome 22, no 3 (1972), 51--79.

\bibitem{drury70}
S.~W. Drury,
\newblock Sur les ensembles de {S}idon.
\newblock {\em C. R. Acad. Sci. Paris S\'er. A-B}, 271 (1970), A162--A163.


\bibitem{dunklrami}
C.~F. Dunkl and D.~E. Ramirez,
\newblock {\em Topics in harmonic analysis}.
\newblock Appleton-Century-Crofts [Meredith Corporation], New York, 1971.
\newblock Appleton-Century Mathematics Series.


\bibitem{enge77}
Ryszard Engelking.
\newblock {\em General topology}.
\newblock PWN---Polish Scientific Publishers, Warsaw, 1977.

\bibitem{eyma}
P. Eymard, \emph{L'~alg\`ebre de {F}ourier d'un groupe
localement compact},
  Bull. Soc. Math. France 92 (1964), 181--236.






  \bibitem{F07} {M. Filali,} {\em On the actions of a
  locally compact group on some of its semigroup compactifications,}
  Math. Proc. Cambridge Philos. Soc. 143 (2007), 25--39.



\bibitem{FG1} {M. Filali \and J. P. Galindo,}
 {\em The size of quotients of function spaces on  a locally compact group} , preprint (2011).

\bibitem{FG2} {M. Filali \and J. P. Galindo,} {\em Strongly prime ultrafilters in semigroup compactifications}, preprint    (2011).


\bibitem{FPr} {M. Filali, Ie. Lutsenko,
 I. Protasov,} {\em Boolean group ideals and the
  ideal structure of $\beta G$,} Mat. Stud. 31 (2009), no. 1, 19--28.


\bibitem{FP} {M. Filali \and J. S. Pym,}
 {\em Right cancellation in the $LUC$-compactification
  of a locally compact group,}  Bull. London Math. Soc.
   35  (2003), 128--134.





\bibitem{fuchsx}
L. Fuchs,
\newblock {\em Infinite abelian groups. Vol. {I}}.
\newblock Academic Press, New York, 1970.




\bibitem{gali10}
J. Galindo,
\newblock On group and semigroup compactifications of topological groups.
\newblock Preprint 2010.



\bibitem{galihern99fu}
J. Galindo and S. Hern\'andez,
\newblock The concept of boundedness and the {B}ohr compactification of a {MAP}
  abelian group.
\newblock {\em Fund. Math.}, 159 (1999), no. 3, 195--218.

\bibitem{GH} {J. Galindo and S. Hern\'andez,}  {\em Interpolation sets and the Bohr topology
of locally compact groups,} Adv. Math. 188 (2004), no. 1, 51--68.

\bibitem{grahhare06i}
C.~C. Graham and K.~E. Hare,
\newblock {$\epsilon$}-{K}ronecker and {$I\sb 0$} sets in abelian groups. {I}.
  {A}rithmetic properties of {$\epsilon$}-{K}ronecker sets.
\newblock {\em Math. Proc. Cambridge Philos. Soc.}, 140  (2006), no.3, 475--489.

\bibitem{grahharekornii}
C.~C. Graham, Kathryn~E. Hare, and Thomas~W. K\"orner,
\newblock {$\epsilon$}-{K}ronecker and {$I\sb 0$} sets in abelian groups. {II}.
  {S}parseness of products of {$\epsilon$}-{K}ronecker sets.
\newblock {\em Math. Proc. Cambridge Philos. Soc.}, 140 (2006), no.3,
491--508.

\bibitem{grahhare05iii}
C.~C. Graham and Kathryn~E. Hare,
\newblock {$\epsilon$}-{K}ronecker and {$I\sb 0$} sets in abelian groups.
  {III}. {I}nterpolation by measures on small sets.
\newblock {\em Studia Math.}, 171 (2005), no.1, 15--32.


\bibitem{grahhare06iv}
C.~C. Graham and Kathryn~E. Hare,
\newblock {$\epsilon$}-{K}ronecker and {$I\sb 0$} sets in abelian groups. {IU}.
  {I}nterpolation by non-negative measures.
\newblock {\em Studia Math.}, 177 (2006), no. 1, 9--24.


\bibitem{gr} A. Grothendieck, Criteres de compacite dans les espaces
fonctionnels generaux. {\em Amer. J. Math.} 74 (1952), 168--186.

\bibitem{hartryll64}
S.~Hartman and C.~Ryll-Nardzewski,
\newblock Almost periodic extensions of functions.
\newblock {\em Colloq. Math.}, 12 (1964), 23--39.


\bibitem{hern08}
S.~ Hern\'andez,
\newblock The {B}ohr topology of discrete nonabelian groups.
\newblock {\em J. Lie Theory}, 18 (2008), no. 3, 733--746.

\bibitem{hewiross1}
E.~Hewitt and K.~A. Ross,
\newblock {\em Abstract harmonic analysis. {U}ol. {I}: {S}tructure of
  topological groups. {I}ntegration theory, group representations}.
\newblock Academic Press Inc., Publishers, New York, 1963.

\bibitem{HS}
N.~Hindman and D.~Strauss, \emph{Algebra in the {Stone--{\v C}ech}
  compactification}, Walter de Gruyter, Berlin, 1998.

\bibitem{kaha70}
J.~P.~Kahane,  \emph{S\'eries de Fourier absolument convergentes.}
 Ergebnisse der Mathematik und ihrer Grenzgebiete, Band 50 Springer-Verlag, Berlin-New York  1970 viii+169 pp.
\bibitem{lopezross}
J.~M. L\'opez and K.~A. Ross,
\newblock {\em Sidon sets}.
\newblock Marcel Dekker Inc., New York, 1975.
\newblock Lecture Notes in Pure and Applied Mathematics, Uol. 13.


\bibitem{pica73b}
M.~A.~Picardello,
\newblock Lacunary sets in discrete noncommutative groups.
\newblock {\em Univ. Genova Pubbl. Ist. Mat. (2)}, 60 (1973), ii+24 pp.

\bibitem{Pym}
J.~Pym, \emph{A note on ${G}^\mathcal{LUC}$ and Veech's theorem},
Semigroup
  Forum 59 (1999), 171--174.

   \bibitem{ra} D.~E.~Ramirez, {\em Weakly almost
   periodic functions and Fourier-Stieltjes transforms,}
   Proc. Amer. Math. Soc. 19 (1968), 1087--1088.

\bibitem{rams80}
L.~T.~Ramsey,
\newblock A theorem of {C}. {R}yll-{N}ardzewski and metrizable l.c.a. groups.
\newblock {\em Proc. Amer. Math. Soc.}, 78 (1980), 221--224.


\bibitem{roeldier}
W.~Roelcke  and S.~Dierolf,
 \newblock {\em Uniform structures on topological groups and their quotients.}
 \newblock
Advanced Book Program.
McGraw-Hill International Book Co., New York.


    \bibitem{ru} {W.~Rudin,}{\em Weak almost periodic functions and
Fourier-Stieltjes transforms,} Duke Math. J. 26 (1959), 215--220.

   \bibitem{rup} {W.~Ruppert,}{\em On
   weakly almost periodic sets,} Semigroup Forum 32 (1985), 267--281.

%


\end{thebibliography}
\end{document}